\documentclass[hidelinks,onefignum,onetabnum]{siamart251216}

\usepackage{amsmath,amssymb}
\usepackage{algpseudocode}
\usepackage{graphicx}
\usepackage{xcolor}
\usepackage{enumitem}
\usepackage{multirow}
\IfFileExists{todonotes.sty}{\usepackage{todonotes}}{}
\usepackage{accents}
\newsiamthm{idea}{Idea}
\newsiamthm{assumption}{Assumption}

\newcommand{\xlb}{\underline{x}}
\newcommand{\xub}{\overline{x}}
\newcommand{\xtil}{\tilde{x}}
\newcommand{\xltil}{\underaccent{\tilde}{x}}
\newcommand{\cA}{\mathcal{A}}
\newcommand{\cO}{\mathcal{O}}

\headers{Universal and Parameter-free Gradient Sliding}{Y. Wu, Y. Ouyang, Z. Zhang, and Q. Luo}

\title{Universal and Parameter-free Gradient Sliding for Composite Optimization\thanks{Submitted to the editors on \today.\funding{Wu, Ouyang, and Luo are partially supported by the AFOSR grant FA9550-25-1-0278; Zhang is partially supported by N00014-20-1-2089}}}
\author{Yan Wu\thanks{Department of Industrial Engineering, Clemson University, Clemson, SC}
\and Yuyuan Ouyang\thanks{School of Mathematical and Statistical Sciences, Clemson University, Clemson, SC}
\and Zhe Zhang\thanks{Edwardson School of Industrial Engineering, Purdue University, West Lafayette, IN}
\and Qi Luo\thanks{Department of Business Analytics, Tippie College of Business, University of Iowa, Iowa City, IA (\email{qi-luo-1@uiowa.edu}).  Corresponding author.}}

\ifpdf
\hypersetup{
  pdftitle={Universal and Parameter-free Gradient Sliding for Composite Optimization},
  pdfauthor={Yan Wu, Yuyuan Ouyang, Zhe Zhang, and Qi Luo}
}
\fi

\begin{document}

\maketitle

\begin{abstract}
We propose a Parameter-Free Universal Gradient Sliding (PFUGS) algorithm for computing an approximate solution to the convex composite optimization $\min_{x\in X} \{f(x) + g(x)\}$, where $f$ has $(M_\nu,\nu)$-H\"older continuous subgradient and $g$ has $L$-Lipschitz continuous gradient. 
PFUGS  computes an $\varepsilon$-approximate solution with $\mathcal{O}((M_\nu/\varepsilon)^{{2}/{(1+3\nu)}})$ evaluations of (sub)gradients of $f$  and $\mathcal{O}((L/\varepsilon)^{1/2})$ evaluations of gradients of $g$,  without  prior knowledge of problem constants. 
To the best of our knowledge, PFUGS is the first gradient sliding algorithm for problems involving two functions whose distinct problem constants are both unknown a priori.
\end{abstract}
\begin{keywords}
composite optimization, parameter-free method, universal gradient sliding 
\end{keywords}
\begin{MSCcodes}
90C25, 90C06, 49M37
\end{MSCcodes}
\section{Introduction}
In this paper, we consider a class of composite problems
\begin{equation}
    \label{eq:problem_of_interest}
    \min_{x\in X} f(x)+g(x),
\end{equation}
where $X$ is a simple closed convex set,
$f: \mathbb{R}^n\to \mathbb{R}$ is convex and possibly non-smooth, $g: \mathbb{R}^n\to \mathbb{R}$ is convex and smooth. In addition, we assume that for a pre-specified norm $\|\cdot\|$, there exists $\nu\in [0,1]$ and $M_\nu, L>0$ such that for all $x,y\in\mathbb{R}^n$,
\begin{align}
    \label{eq:holder_descent}
    & f(y) - f(x) - \langle f'(x), y - x\rangle \le \tfrac{M_\nu}{1+\nu}\|y -x\|^{1+\nu},\ \forall f'(x)\in\partial f(x),\text{ and }
    \\
    \label{eq:L}
    & g(y) - g(x) - \langle \nabla g(x), y - x\rangle \le \tfrac{L}{2}\|y -x\|^2.
\end{align}

In convex optimization literature, condition \eqref{eq:holder_descent} is known as the H\"older condition for the (sub)gradient $f'$ with constant $M_\nu$ and exponent $\nu$ (see, e.g., \cite{nesterov2015universal,lan2015bundle}). This assumption covers cases in which $f$ could behave as a nonsmooth function (with $\nu=0$), a smooth function (with $\nu=1$), or exhibit intermediate behavior between nonsmooth and smooth functions (with $\nu\in (0,1)$); the latter case is also referred to as weakly smooth. Condition \eqref{eq:L} is known as the Lipschitz condition for the gradient $\nabla g$ with constant $L$. In some literature, it is also called the $L$-smoothness condition.

We study the composite convex optimization problem due to its broad applicability.
For example, in sparse regression and total variation image reconstruction, $g$ represents a data fidelity term, while $f$ encodes structural regularization regarding prior information about desired solutions. 
When $f$ is proximal-friendly,
the problem has been extensively studied  \cite{nesterov2012gradient}. 
In this setting, algorithms for solving problem \eqref{eq:problem_of_interest} are often referred to as proximal algorithms
\cite{parikh2014proximal}, and 
an $\varepsilon$-solution can be obtained within $\cO(\sqrt{L/\varepsilon})$ gradient evaluations of $g$, which matches the complexity bound for unconstrained smooth  convex optimization.
Hence, adding a proximal-friendly term $f$ does not increase the first-order oracle complexity.
Many parameter-free proximal algorithms for solving \eqref{eq:problem_of_interest} without requiring \emph{prior knowledge} of problem parameters  have been proposed.
In particular, \cite{nesterov2015universal} 
generalizes the assumption on $g$ by relaxing $L$-smoothness to 
H\"older continuity of its (sub)gradient, and proposes a parameter-free proximal algorithm that achieves $\cO((L_\nu/\varepsilon)^{2/(1+3\nu)})$ (sub)gradient evaluations of $g$, where $\nu$ and $L_\nu$ denote the exponent and constant in the  H\"older condition.

The problem becomes significantly more challenging when $f$ is no longer proximal-friendly.
In this case, standard first-order methods with 
linear approximation of $f$
often yield unfavorable (sub)gradient evaluation complexity. 
For example, when $f$ is nonsmooth with $\nu=0$, the subgradient method 
requires $\cO((M_{\nu=0}+L)^2/\varepsilon^2)$ (sub)gradient evaluations of both $f$ and $g$.
The Gradient Sliding (GS) method addresses this issue by 
 exploiting the composite structure of \eqref{eq:problem_of_interest}  and carefully balancing the computational effort devoted to evaluating (sub)gradients of $f$ and $g$:
(1) When $f$ is nonsmooth and Lipschitz continuous with constant $M_{\nu=0}$, the GS method in \cite{lan2016gradient} 
achieves $\cO(\sqrt{L/\varepsilon})$ gradient evaluations of $g$ and $\cO(({M_{\nu=0}/\varepsilon)^2})$ subgradient evaluations of $f$;
(2) When $f$ is smooth with constant $M_{\nu=1}$, an accelerated GS in \cite{lan2022accelerated} 
achieves $\cO(\sqrt{L/\varepsilon})$ gradient evaluations of $g$ and $\cO(\sqrt{M_{\nu=1}/\varepsilon})$ gradient evaluations of $f$. 
These results are obtained by separating the (sub)gradient complexities associated with the two functions in the objective, yielding optimal complexity bounds for the respective function classes.
The central idea of GS is,  at each iteration, linearizing a fixed component in the objective while approximately solving the resulting subproblem with respect to the other component via an inner first-order method.
The key design is the
choice of the number of inner iterations: 
too few leads to excessive gradient evaluations of $g$, while too many increases the computational cost of $f$.
A central contribution of GS methods 
is a proper design of number of inner iterations based on problem constants, thereby attaining optimal and separated (sub)gradient complexities for both components.
As a result, when problem parameters are known, GS can substantially outperform standard methods for composite convex optimization.

There has been a growing interest in parameter-free first-order methods, namely, algorithms that adjust to unknown problem constants 
while retaining strong complexity guarantees. Throughout this paper we use the term ``parameter-free'' to describe any algorithm that does not need knowledge of any problem constants (e.g. H\"older constant and exponent) other than the desired accuracy threshold $\varepsilon$. This line of work includes
methods based on Lipschitz constant/H\"older constants/local curvature estimation \cite{nesterov2015universal,malitsky2020adgd,malitsky2024adapgm,lan2026optimal,li2025simple,lan2024projected,lan2024acpdhg,deng2026uniformly,suh2025adanag,borodich2025graal,lu2023pfcg,lu2023llcg,lu2025pde} and bundle-type or bundle-level-type methods \cite{monteiro2016adaptive,sujanani2025parameterfree,guigues2026universal,guigues2024adaptivebundle,monteiro2024pfbundle,ChenLanOuyangZhang2019bundle,lan2015bundle,chen2020acceleration}. 
Nevertheless, existing works do not address the challenge of balancing computational effort across multiple objective components in \eqref{eq:problem_of_interest} with
different and unknown problem constants. This issue is acute for parameter-free GS methods, where inner and outer iterations must be coordinated 
to preserve the desired separated (sub)gradient evaluation complexities.

\subsection{Challenges of parameter-free GS and our contributions}
Existing GS methods attain optimal complexity bounds only when  problem constants are known a priori.
To the best of our knowledge, the only parameter-free GS-type method with optimal complexity bounds is \cite{ouyang2023universal}, which 
assumes that only one function has an unknown problem constant. 
Our work aims to develop the first fully parameter-free GS method for composite optimization. 

Designing such methods presents two main challenges. 
First, incorporating backtracking linesearch into GS is nontrivial.
In standard first-order optimization, e.g., Accelerated Gradient Descent (AGD), 
backtracking linesearch adaptively estimates one local constant per iterate.
In the GS setting, however, the parameter estimates associated with different components in the objective may evolve at different rates across iterations.
Since optimal complexities rely critically on choosing the appropriate number of inner iterations, maintaining the required balance 
becomes difficult when both estimates change dynamically.
Second, failed backtracking steps are substantially more costly in GS. 
A failed trial in standard first-order optimization incurs only a constant amount of extra (sub)gradient evaluations.
By contrast, in the GS framework, each failed outer constant trial may involve a large number of inner iterations. Specifically, 
if the outer constant is underestimated, then the resulting number of inner iterations may be severely overestimated, and such linesearch failure may destroy the desired complexity bound for the inner component.
On the other hand, if the outer constant is overestimated, the overly conservative estimate leads to 
unnecessarily small progress per outer iteration and thus degrades the overall efficiency. 

In this paper, we address the above issues and develop the first parameter-free GS framework for convex composite optimization.
Our contributions are threefold.
First, we propose a new linesearch-compatible sliding subroutine that determines the number of inner iterations dynamically rather than fixing it in advance.
This allows the method to adapt to the varying local estimates while preserving the balance required by the GS scheme.
Second, we 
combine this subroutine with outer backtracking linesearch. We first identify necessary conditions under which a naive backtracking strategy preserves the desired complexity bounds.
We then develop a new outer backtracking strategy that is \emph{fully} parameter-free and requires no prior information about any problem constants.
This strategy is designed to avoid both under- and over-estimation of the constants, 
which preserves the balance between the outer and inner computational efforts,
recovering the optimal complexity bounds of GS with known parameters.
Third, by incorporating the universal gradient method in  
\cite{nesterov2015universal}, we develop a Parameter-Free Universal Gradient Sliding (PFUGS) algorithm in the most general setting, where 
$f$ can be nonsmooth ($\nu=0$), smooth ($\nu=1$), or weakly smooth ($\nu\in(0,1)$) without requiring any knowledge of problem constants. 

\subsection{Organization and main results}
The remainder of the paper is organized as follows.
Section~\ref{sec:GDS} revisits the GS method under a simplified assumption that $f$ is $M$-smooth, which 
highlights the main technical challenges in the unknown-parameter setting.
Section~\ref{sec:GDS_linesearch} illustrates our strategies and develops a parameter-free GS framework via a modified sliding subroutine and outer linesearch schemes.
Section \ref{sec:universal} extends these parameter-free strategies to more general 
settings and includes acceleration to achieve optimal complexities. 
Section \ref{sec:PFUGS} proposes the full parameter-free algorithm PFUGS and demonstrates its optimal complexity.
Section~\ref{sec:numerical_experiments} presents numerical experiments demonstrating the performance of PFUGS. 
Section \ref{sec:conclusion} concludes our paper with potential directions for future research.
For better readability, some of the proofs are postponed to the appendix.

\section{Preliminary: Gradient Descent Sliding (GDS)}\label{sec:GDS}
In this section, we review the GS technique
through a simplified gradient-descent-sliding scheme and state at a high level the difficulties that arise when the problem constants of $f$ and $g$ are unknown.
The material and analysis closely follow \cite{lan2016gradient}; we intentionally restrict attention to non-accelerated gradient-descent-type schemes to keep the exposition compact. Specifically, we consider the following algorithmic design question: suppose that $f$ is $M$-smooth and $g$ is $L$-smooth. 
Naive gradient descent requires $\cO((M+L)/\varepsilon)$ gradient evaluations of both $f$ and $g$ to compute an $\varepsilon$-solution. If $M\ge L$, this may introduce an additional and potentially unnecessary $\cO((M+L)/\varepsilon)$ gradient evaluations of $\nabla g$. Can we design a scheme that reduces the number of $\nabla g$  evaluations to $\cO(L/\varepsilon)$? Moreover, can we do so in a parameter-free manner, without prior knowledge of $M$ and $L$, by estimating them on the fly? In this section, we revisit the answer in \cite{lan2016gradient} to the first question and highlight the technical challenges in the second question. To the best of our knowledge, the second question is still open in the literature.

To keep the exposition compact and focused, we make the following simplifications in this section and Section \ref{sec:GDS_linesearch}.
We assume that $\|\cdot\|$ is the Euclidean norm and that $f$ is $M$-smooth, which is consistent with the H\"older \eqref{eq:holder_descent} assumption with exponent $\nu=1$ and constant $M_{\nu=1}$.
For notational convenience,
we drop the subscript in $M$ in these two sections. 
We also assume, without loss of generality, that \(M\ge L\), since the case
\(M\le L\) can be handled by swapping the roles of the two functions.
We refer to the method described in Algorithm \ref{alg:GDS} as GDS. 
To streamline the exposition, we omit the proofs of Lemmas~\ref{thm:GDS_outer_convergence} and~\ref{thm:GDS_subroutine_convergence}, which can be derived from  \cite{lan2016gradient}. Complete proofs are provided in the arXiv version\footnote{\url{https://arxiv.org/abs/2603.23492}} of this paper.
\begin{algorithm}[h]
\caption{Main Procedure for GDS}
\label{alg:GDS}
\begin{algorithmic}[1]
\State Initialization: $\tilde x_0 = x_0\in X$
\For{$k=1,2,\dots, N$}
\State Set $\eta_k = L_k$, approximate $g(x)$ by $\phi_k(x):=g(\xtil_{k-1})+\langle \nabla g(\tilde x_{k-1}), x-\xtil_{k-1}\rangle + ({\eta_k}/{2})\|x - \tilde x_{k-1}\|^2$,
and call subroutine $(x_k, \tilde x_k)=\mathcal{A}(x_{k-1}, \nabla g(\xtil_{k-1}), \eta_{k}, T_k)$ to compute an approximate solution $(x_k, \tilde x_k)$ to the subproblem
\begin{align}
    \label{eq:GDS_subproblem}
    &\min_{x\in X} \phi_k(x) + f(x)
\end{align}
with approximate optimality condition $\langle \nabla\phi_k(\xtil_k), \xtil_k - x\rangle + f(\xtil_k) - f(x) \le \varepsilon_k(x)$ for all $x\in X$\label{state:GDS_subproblem_optimality}
\EndFor
\State Output
$\xub_N = (\sum\nolimits_{k=1}^{N} 1/\eta_k)^{-1}\sum\nolimits_{k=1}^{N} \tilde x_k/\eta_k$
\end{algorithmic}
\end{algorithm}

In Algorithm \ref{alg:GDS}, $\varepsilon_k(x)$ 
denotes the accuracy in solving subproblem \eqref{eq:GDS_subproblem}. 
If $\varepsilon_k(x)\equiv 0$, $\xtil_k$ is an exact solution of \eqref{eq:GDS_subproblem} and Algorithm \ref{alg:GDS} reduces to the proximal gradient method. 
The convergence property of Algorithm \ref{alg:GDS} is stated below.
\begin{lemma}\label{thm:GDS_outer_convergence}
Suppose that the iterates $\tilde x_k$'s in Algorithm \ref{alg:GDS} satisfy for any $k\ge1$,
\begin{align}
&g(\xtil_k)
    \le
    g(\tilde x_{k-1})
    + \langle \nabla g(\tilde x_{k-1}), \tilde x_k - \tilde x_{k-1} \rangle
    + \tfrac{L_k}{2}\|\tilde x_k - \tilde x_{k-1}\|^2.
    \label{eq:GDS_g_ineq}
\end{align}
Then, for any $x\in X$, we have
\begin{equation*}
    \begin{aligned}
    & g(\xub_N) + f(\xub_N)-g(x)-f(x)
    \\
    \le &
    (\sum\nolimits_{k=1}^N\tfrac{1}{\eta_k})^{-1}
    \sum\nolimits_{k=1}^N\tfrac{\varepsilon_k(x)}{\eta_k}
    +
    (2\sum\nolimits_{k=1}^N\tfrac{1}{\eta_k})^{-1}
    (\|x-\xtil_{0}\|^2 -\|x-\xtil_N\|^2).
     \end{aligned}
\end{equation*}
\end{lemma}

The above lemma shows that, with properly chosen stepsize $\eta_k$ and subproblem error $\varepsilon_k(x)$, the output $\xub_N$ becomes an $\varepsilon$-solution.
To solve subproblem \eqref{eq:GDS_subproblem} with a suitable error $\varepsilon_k(x)$, 
Algorithm~\ref{alg:GDS} invokes
subroutine $\mathcal A$.
In view of the structure of subproblem \eqref{eq:GDS_subproblem}, it is natural to employ a proximal gradient method as the subroutine. Algorithm~\ref{alg:GDS_subroutine} presents one such implementation of $\mathcal A$.
\begin{algorithm}[h]
    \caption{Proximal-Gradient Subroutine 
    $\mathcal{A}$
    }
    \label{alg:GDS_subroutine}
    \begin{algorithmic}[1]
        \State Input: $x_{k-1}, \nabla g(\xtil_{k-1}), \eta_{k}, T_k$
        \State Set $x_k^0 = x_{k-1}$
        \For{$t=1,\dots,T_k$}
        \State Set $p_k^t = M_k^t$ and compute
        \vspace{-5pt}
\begin{equation}\label{eq:GDS_subroutine_update}
            x_k^t = \arg\min \limits_{x\in X} \langle \nabla g(\xtil_{k-1}) + \nabla f(x_k^{t-1}), x\rangle + \tfrac{\eta_k}{2}\|x-\xtil_{k-1}\|^2+\tfrac{p_k^t}{2}\|x-x_k^{t-1}\|^2
        \end{equation}
        \vspace{-10pt}
        \EndFor
        \State Output $x_k = x_k^{T_k}$ and $\xtil_k = (\sum\nolimits_{\tau=1}^{T_k} 1/p_k^\tau )^{-1}\sum\nolimits_{t=1}^{T_k} x_k^t/p_k^t $
    \end{algorithmic}
\end{algorithm}

The main idea of GS is already evident from Algorithms~\ref{alg:GDS} and \ref{alg:GDS_subroutine}. Specifically, the total numbers of gradient evaluations of $\nabla g$ and $\nabla f$ are $N$ and $\sum_{k=1}^{N} T_k$, respectively. Suppose that the Lipschitz constants $L$ and $M$ are known, and that $L_k\equiv L$ and $M_k^t\equiv M$. If, ideally, $T_k\equiv\infty$, then $\varepsilon_k(x)\equiv 0$, and Algorithm~\ref{alg:GDS} reduces to a proximal gradient method with $\cO(L/\varepsilon)$ evaluations of $\nabla g$ but infinitely many gradient evaluations of $\nabla f$. If instead $T_k\equiv1$, then $x_k=\xtil_k$, and Algorithm~\ref{alg:GDS} reduces to projected gradient descent with gradient complexity $\cO((M+L)/\varepsilon)$ for both $\nabla g$ and $\nabla f$. The goal of GS is therefore to choose $T_k>1$ appropriately so as to attain gradient complexity $\cO(L/\varepsilon)$ for $\nabla g$ while preserving gradient complexity $\cO((M+L)/\varepsilon)$ for $\nabla f$. To determine such a choice of $T_k$, we next study the convergence behavior of Subroutine~\ref{alg:GDS_subroutine} and its dependence on $T_k$.
\begin{lemma}\label{thm:GDS_subroutine_convergence}
Fix any $x\in X$, and $k\ge1$. Suppose that iterates $x_k^t$'s in Subroutine \ref{alg:GDS_subroutine} satisfy for any $t\ge1$,
\begin{equation}\label{eq:GDS_f_ineq}
\begin{aligned}
    &f(x_k^t)
    \le
    f(x_k^{t-1})
    + \langle \nabla f(x_k^{t-1}), x_k^t - x_k^{t-1} \rangle
    + \tfrac{M_k^t}{2}\|x_k^t - x_k^{t-1}\|^2.
\end{aligned}
\end{equation}
The approximate solution $(x_k,\tilde x_k)$ produced by Subroutine \ref{alg:GDS_subroutine} satisfies the approximate optimality condition described in Algorithm \ref{alg:GDS} with $\varepsilon_k(x) = (2\sum\nolimits_{\tau=1}^{T_k} 1/p_k^\tau )^{-1}(\|x-x_{k-1}\|^2-\|x-x_k\|^2)$ after $T_k$ iterations.
\end{lemma}

With the help of Lemmas \ref{thm:GDS_outer_convergence} and \ref{thm:GDS_subroutine_convergence},
we can now state the overall convergence of Algorithm \ref{alg:GDS} with Subroutine \ref{alg:GDS_subroutine}.
\begin{theorem}\label{thm:GDS_complexity}
    Under conditions of Lemmas \ref{thm:GDS_outer_convergence} and \ref{thm:GDS_subroutine_convergence}, we have
\begin{equation}
\label{eq:GDS_general_complexity}
\begin{aligned}
        & g(\xub_N) + f(\xub_N)-g(x)-f(x)
        \\
    \le &
    (\sum\nolimits_{k=1}^N\tfrac{1}{\eta_k})^{-1}
    \sum\nolimits_{k=1}^N(2\eta_k\sum\nolimits_{\tau=1}^{T_k} \tfrac{1}{p_k^\tau} )^{-1}(\|x-x_{k-1}\|^2-\|x-x_k\|^2)
    \\
    &+ (2\sum\nolimits_{k=1}^N\tfrac{1}{\eta_k})^{-1}
    (\|x-\xtil_{0}\|^2 -\|x-\xtil_N\|^2).
\end{aligned}
    \end{equation}
    If $T_k$ is chosen so that $\eta_k\sum\nolimits_{\tau=1}^{T_k} 1/p_k^\tau =1/a$ remains constant for some $a>0$, then
\begin{equation}\label{eq:GDS_Tk1_complexity}
        g(\xub_N) + f(\xub_N)-g(x)-f(x)
        \le
        (2\sum\nolimits_{k=1}^N\tfrac{1}{\eta_k})^{-1}(a+1)\|x-x_0\|^2.
    \end{equation}
\end{theorem}
\begin{proof}
Bound \eqref{eq:GDS_general_complexity} follows directly from Lemmas~\ref{thm:GDS_outer_convergence} and \ref{thm:GDS_subroutine_convergence}.
We conclude this theorem by recalling  $x_0=\tilde{x}_0$ and dropping some negative terms in \eqref{eq:GDS_general_complexity}.
\end{proof} 

To better understand the role of $T_k$, it is instructive to examine two extreme choices when the Lipschitz constants $L$ and $M$ are known.
If $T_k\equiv\infty$, the resulting method has $\cO(L/\varepsilon)$ gradient complexity for $\nabla g$, at the expense of infinitely many gradient evaluations of $\nabla f$. If instead $T_k\equiv1$, both $\nabla g$ and $\nabla f$ admit gradient complexity of $\cO((M+L)/\varepsilon)$. Both statements follow directly from the above theorem. The corollary below identifies an optimal choice of $T_k$ in the GS framework. Throughout the rest of the paper, we use $x^*$ to denote any optimal solution of problem~\eqref{eq:problem_of_interest}.
\begin{corollary}\label{thmGDS_parameter_know_constant}
Let the parameters of Algorithm \ref{alg:GDS} and Subroutine \ref{alg:GDS_subroutine} be set to 
\(L_k=L\), \(T_k=\lceil M/L\rceil\),
and 
\(M_k^t=M, \forall k\ge1, t\ge1\). To obtain an $\varepsilon$-solution, it requires at most $\mathcal{O}( L\|x^*-x_0\|^2/\varepsilon)$ and $\mathcal{O}((M+L)\|x^*-x_0\|^2/\varepsilon)$ evaluations of $\nabla g$ and $\nabla f$, respectively.
\end{corollary}
\begin{proof}
\eqref{eq:GDS_g_ineq} and \eqref{eq:GDS_f_ineq} hold when the Lipschitz constants $M$ and $L$ are known.
Moreover, $\eta_k\sum\nolimits_{\tau=1}^{T_k}1/p_k^\tau = (L/M)\lceil M/L\rceil\ge 1$,
hence Theorem \ref{thm:GDS_complexity} holds with $a\le 1$. 
Because $\sum\nolimits_{k=1}^N{1}/{\eta_k} = N/L$ and \eqref{eq:GDS_Tk1_complexity}, we have
\(
g(\xub_N)+f(\xub_N)-g(x^*)-f(x^*)\le L\|x^*-x_0\|^2/N\le \varepsilon
\) if we set $N =\lceil L\|x^*-x_0\|^2/\varepsilon\rceil$. 
Recalling that $T_k = \lceil M/L\rceil$, we have
$\sum\nolimits_{k=1}^{N}T_k =N\lceil M/L\rceil \le (M+L)N/L$. 
\end{proof}

The above corollary highlights the competitiveness of GS when the Lipschitz constants $L$ and $M$ are known.
However, because these problem-dependent constants are rarely known in practice, we need to set $\eta_k=L_k$ and $p_k^t=M_k^t$, where $L_k$ and $M_k^t$ are obtained through backtracking linesearch.
This setting leads to two algorithmic design challenges.
First, the preceding analysis relies on the telescoping structure in \eqref{eq:GDS_general_complexity}, which is the most convenient to exploit when
$\eta_k\sum\nolimits_{\tau=1}^{T_k} 1/p_k^\tau$
is monotone in $k$. 
This is precisely why
$\eta_k\sum\nolimits_{\tau=1}^{T_k} 1/p_k^\tau$ is assumed to be constant in deriving the bound \eqref{eq:GDS_Tk1_complexity}.
However, because the sequences $\{L_k\}$ and $\{M_k^t\}$ may vary with both $k$ and $t$, it becomes challenging to choose the number of iterations $T_k$ in subroutine $\cA$ so that
\(
L_k\sum\nolimits_{t=1}^{T_k}1/M_k^t
\)
remains invariant while still preserving the desired gradient complexity.
Second, when performing linesearch on the estimate of $L$, some trial values of $L$ are inevitably incorrect.
In such cases, we must backtrack and resolve subproblem \eqref{eq:GDS_subproblem} with a new estimate of $L$.
Since each backtracking step incurs an additional call to Subroutine~\ref{alg:GDS_subroutine}, we must ensure 
that the overall complexity is preserved. 
\section{Parameter-free strategy for GDS}\label{sec:GDS_linesearch}
This section addresses the above two critical challenges when problem constants $L$ and $M$ are unknown, and proposes parameter-free strategies for GDS to attain $\cO(L/\varepsilon)$ and $\cO(M/\varepsilon)$ complexity bounds.
\subsection{A modified subroutine for inner backtracking linesearch} The proposed \label{sec:subroutine}Subroutine~\ref{alg:GDS_linesearch_subroutine} modifies Subroutine~\ref{alg:GDS_subroutine} by adding a linesearch strategy for estimating $M$ and proposing a dynamic way to determine the number of inner iterations which addresses the first challenge mentioned above.
\begin{algorithm}[h]
 \caption{A modified subroutine 
$\mathcal{A}$ for parameter-free GDS}
\label{alg:GDS_linesearch_subroutine}
\begin{algorithmic}[1]
\State Input:
$x_{k-1}$,
$\nabla g(\xtil_{k-1})$,
$\eta_k^s$, and $M_k^{0,s}$ such that $M_k^{0,s}/\eta_k^s$ is a positive integer
\State
Set
$x_k^{0,s} = x_{k-1}$, $T_k^{0,s} = M_k^{0,s}/\eta_k^s$, and $t=1$
\State 
Compute
\vspace{-10pt}
\begin{equation*} 
    x_k^{t,s} = \text{argmin}_{x\in X} \langle \nabla g(\xtil_{k-1}) + \nabla f(x_k^{t-1,s}), x\rangle + \tfrac{\eta_k^s}{2}\|x-\xtil_{k-1}\|^2+\tfrac{p_k^{t,s}}{2}\|x-x_k^{t-1,s}\|^2,
\end{equation*}

\vspace{-10pt}
\noindent where $p_k^{t,s}=M_k^{t,s}$,  $M_k^{t,s}=2^{i_k^{t,s}}M_k^{t-1,s}$, and $i_k^{t,s}$ is the smallest nonnegative integer such that
\( 
    f(x_k^{t,s})
    \le 
    f(x_k^{t-1,s})+\langle\nabla f(x_k^{t-1,s}),x_k^{t,s}-x_k^{t-1,s}\rangle+\tfrac{M_k^{t,s}}{2}\|x_k^{t,s}-x_k^{t-1,s}\|^2
\)
\label{state:xkts_line} 
\State Set $T_k^{t,s}=(M_k^{t,s}/M_k^{t-1,s})(T_k^{t-1,s}-t+1)-1+t$

\State If $t = T_k^{t,s}$, terminate and output 
$\xtil_k^s = (\sum\nolimits_{\tau=1}^{T_k^s} 1/p_k^{\tau,s} )^{-1}\sum\nolimits_{t=1}^{T_k^s} x_k^{t,s}/p_k^{t,s} $, $x_k^s = x_k^{T_{k}^s,s}$, $M_k^s = M_k^{T_k^s,s}$; otherwise, set $t\gets t+1$, return to line \ref{state:xkts_line}
\end{algorithmic}
\end{algorithm}

There are three major differences between Subroutine~\ref{alg:GDS_linesearch_subroutine} and \ref{alg:GDS_subroutine}.
First, Subroutine~\ref{alg:GDS_linesearch_subroutine} starts from an initial guess of $M_k^{0,s}$ and outputs an updated estimate $M_k^s\ge M_k^{0,s}$, without requiring the exact value of $M$. 
Second, instead of taking $\eta_k$ as input, it uses $\eta_k^s$, since multiple trial estimates of $L$ may be needed before a suitable value is identified through outer backtracking.
The index $s$ labels these trial estimates $L_k^s$ and the corresponding iterates $x_k^{t,s}$, $x_k^s$, and $\tilde x_k^s$.
Third, the Subroutine~\ref{alg:GDS_linesearch_subroutine} no longer requires a pre-specified iteration count $T_k$ and instead updates it adaptively.
As shown below, this guarantees that
\(
\eta_k^s\sum\nolimits_{t=1}^{T_k^s}1/p_k^{t,s}
\)
stays invariant for all $k$ and $s$, which is essential for recovering the telescoping structure in \eqref{eq:GDS_general_complexity}.
\begin{lemma}
\label{thm:GDS_linesearch_maintain_eq}
    Fix any $k\ge1$ and $s\ge1$. Subroutine~\ref{alg:GDS_linesearch_subroutine} generates positive integers $T_k^{t,s}$ for all $t$, such that
    \(
    \eta_k^s \sum_{t=1}^{T_k^s} 1/p_k^{t,s} = 1
    \)
    holds at the last iteration $T_k^s$.
    \end{lemma}
\begin{proof}
From the description of Subroutine~\ref{alg:GDS_linesearch_subroutine} it is straightforward to observe that $T_k^{t,s}$ is always a well-defined positive integer.
Since $p_k^{t,s}=M_k^{t,s}$ for $t\ge1$, it suffices to prove that $\sum\nolimits_{t=1}^{T_k^s}1/M_k^{t,s}={1}/{\eta_k^s}$.
For $t\ge1$, define $B_k^{t,s} := (T_k^{t-1,s}-t+1)/M_k^{t-1,s}$.
From the definitions of $T_k^{t,s}$ and $B_k^{t,s}$ we can observe that
\(
B_k^{t+1,s}
= B_k^{t,s} - {1}/{M_k^{t,s}}.
\)
Moreover, observe from the termination criteria in Subroutine~\ref{alg:GDS_linesearch_subroutine}
that $B_k^{t+1,s}=0$ when $t$ is the index at termination.
Summarizing the above two observations we conclude that
  $\sum\nolimits_{t=1}^{T_k^s}(1/M_k^{t,s}) = B_k^{1,s}=T_k^{0,s}/M_k^{0,s}=1/\eta_k^s$.
\end{proof}

We conclude this subsection by noting that the convergence analysis of Subroutine~\ref{alg:GDS_linesearch_subroutine} follows directly from that of Subroutine~\ref{alg:GDS_subroutine},
since the two subroutines follow the same core structure, differing only by the additional index $s$.
Specifically, the backtracking procedure in Subroutine~\ref{alg:GDS_linesearch_subroutine} enforces the smoothness condition \eqref{eq:GDS_f_ineq} 
for all $k,s,t\ge1$. 
Therefore, 
Lemma~\ref{thm:GDS_subroutine_convergence} applies directly, with the added superscript $s$.
\subsection{Proposed parameter-free GDS}
The second challenge mentioned at the end of Section~\ref{sec:GDS} is to ensure that the overall complexity bounds are preserved despite the additional gradient evaluations of $f$ incurred when estimating the constant $L$. To address this challenge,
we first consider a naive approach in a special case of Algorithm \ref{alg:GDS} with $\varepsilon_k(x)\equiv 0$. 
In this case, this algorithm reduces to the proximal gradient method.
A natural backtracking linesearch strategy for estimating $L$ starts
from an initial estimate $L_0\le L$ and, at each iteration $k$, tests a trial value $L_k$, 
doubling it until the Lipschitz condition \eqref{eq:GDS_g_ineq} is satisfied.
Motivated by this standard approach, Algorithm \ref{alg:GDS_linesearch_naive} presents a naive parameter-free implementation of GDS.
\begin{algorithm}[h]
\caption{GDS with naive backtracking linesearch}
\label{alg:GDS_linesearch_naive}
\begin{algorithmic}[1]
\State Initialization: $\xtil_0=x_0\in X$, $L_0$ such that $\varepsilon/\|x_0 - x^*\|^2\le L_0 \le L$ and $M_0=L_0$
\For{$k=1,2,\dots, N$}
\State Set $s=1$, and $M_k^0 = M_{k-1}$
\State
Compute
$L_k^s = 2^{s-1}L_{k-1}$,
$\eta_k^s = L_k^s$, and
$M_k^{0,s} = \max\{M_k^{s-1}, \eta_k^s\}$
\label{state:lks_update}
\State Call subroutine $(x_k^s, \xtil_k^{s},M_k^s) = \mathcal{A}(x_{k-1}, \nabla g(\xtil_{k-1}), \eta_k^s, M_k^{0,s})$ in Algorithm~\ref{alg:GDS_linesearch_subroutine}
\State If $g(\xtil_k^s)
    >
    g(\tilde x_{k-1})
    + \langle \nabla g(\tilde x_{k-1}), \tilde x_k^s - \tilde x_{k-1} \rangle
    + \tfrac{L_k^s}{2}\|\tilde x_k^s - \tilde x_{k-1}\|^2,$ update $s\gets s+1$ and return to line \ref{state:lks_update}; otherwise, set
$S_k = s$,
$L_k=L_k^{S_k}$,
$x_k= x_k^{S_k}$,
$\xtil_k = \xtil_k^{S_k}$,
and $M_k = M_k^{S_k}$ \label{state:k1_outer_ineq}
\EndFor

\State Output:
$\xub_N =(\sum\nolimits_{k=1}^{N} 1/\eta_k^{S_k})^{-1}\sum\nolimits_{k=1}^{N} \xtil_k/\eta_k^{S_k} $
\end{algorithmic}
\end{algorithm}

In Algorithm \ref{alg:GDS_linesearch_naive}, we require that the initial guess $L_0$ satisfies $\varepsilon/\|x_0 - x^*\|^2\le L_0 \le L$. Here the condition that $L_0\le L$ can be achieved by selecting two points $y$ and $z$ in $X$ at random and setting $L_0= 2(g(z)-g(y)-\langle \nabla g(y), z-y\rangle)/\|y-z\|^2$. The requirement that $L_0\ge \varepsilon/\|x_0 - x^*\|^2$ is necessary for the complexity analysis, as we will see in the proof of Corollary \ref{thm:GDS_naive_corollary}. This is a drawback of the naive linesearch approach that we will address later.

Similar to the discussion after Lemma~\ref{thm:GDS_linesearch_maintain_eq}, due to the analogy between Algorithms \ref{alg:GDS} and \ref{alg:GDS_linesearch_naive}, Lemma~\ref{thm:GDS_outer_convergence} and Theorem \ref{thm:GDS_complexity} apply to Algorithm \ref{alg:GDS_linesearch_naive} directly with $\eta_k:=\eta_k^{S_k}$ and $T_k := T_k^{S_k}$.
Then the complexity of Algorithm~\ref{alg:GDS_linesearch_naive} can be stated below.
\begin{corollary}\label{thm:GDS_naive_corollary}
    Algorithm \ref{alg:GDS_linesearch_naive}
     with Subroutine~\ref{alg:GDS_linesearch_subroutine} requires at most $\lceil2L\|x_0-x^*\|^2/\varepsilon\rceil$ and $\lceil 8 M\|x_0-x^*\|^2/\varepsilon\rceil$ evaluations of $\nabla g$ and $\nabla f$  to
    compute an $\varepsilon$-solution.
\end{corollary}
\begin{proof} 
We start by noting that the integrality requirement of \(M_k^{0,s}/\eta_k^s\) stated in Subroutine~\ref{alg:GDS_linesearch_subroutine} is well defined. Indeed, for every \((k,s)\), the
definition of \(M_k^{0,s}\) and the search rule for \(L_k^s\) and \(M_k^{t,s}\) ensure that
\(M_k^{0,s}/\eta_k^s\) is a positive integer.   
We now analyze the complexities. Noting from the descriptions of Algorithm~\ref{alg:GDS_linesearch_naive} and Subroutine~\ref{alg:GDS_linesearch_subroutine}, we have $\eta_k^s = L_k^s$ and $p_k^{t,s}=M_k^{t,s}\le 2M$. 
Then,
utilizing the above relations and  Lemma~\ref{thm:GDS_linesearch_maintain_eq}  
we have
        \(\sum\nolimits_{s=1}^{S_k} T_k^s
        \le
        \sum\nolimits_{s=1}^{S_k} \sum\nolimits_{t=1}^{T_k^s}2M/M_k^{t,s}
        =    2M\sum\nolimits_{s=1}^{S_k} 1/L_k^s \le 4M/L_k^1 =  4M/L_{k-1}
        \).
        Let $N_\varepsilon\ge1$ be the first outer iteration index such that 
        \(\sum\nolimits_{k=1}^{N_\varepsilon}1/L_k\ge \|x_0 - x^*\|^2/\varepsilon 
        > \sum\nolimits_{k=1}^{N_\varepsilon-1}1/L_k 
        \). 
        Then, by   Lemma~\ref{thm:GDS_linesearch_maintain_eq} and the relation $\eta_k=L_k \le 2L$, we have from Theorem \ref{thm:GDS_complexity} that $\xub_{N_\varepsilon}$ is an $\varepsilon$-solution and the required number of evaluations of $\nabla g$ is 
        \(
        N_\varepsilon\le \lceil 2L\|x_0-x^*\|^2/\varepsilon\rceil
        \).
        Recalling that $L_0\ge \varepsilon/\|x_0 - x^*\|^2$, it then follows from the above analysis that the needed number of evaluations of $\nabla f$ is 
\(      \sum\nolimits_{k=1}^{N_\varepsilon}\sum\nolimits_{s=1}^{S_k}T_k^s
        \le 4M  ( {1}/{L_0}+\sum\nolimits_{k=1}^{N_\varepsilon-1} {1}/{L_k}  )
        \le 8M 
        \|x_0-x^*\|^2/\varepsilon.
\)
\end{proof}

The above corollary shows that our proposed GDS with naive linesearch requires at most $\cO(L\|x_0-x^*\|^2/\varepsilon)$ and $\cO( M\|x_0-x^*\|^2/\varepsilon)$ evaluations of $\nabla g$ and $\nabla f$ respectively to compute an $\varepsilon$-solution. It is also easy to show that the total number of function evaluations of $g$ and $f$ are in the same order.
Note that although Algorithm \ref{alg:GDS_linesearch_naive} is motivated by the standard backtracking linesearch strategy in proximal gradient methods, the above analysis is nontrivial due to
the cost incurred by incorrect estimates of $L$.
In standard proximal gradient descent methods, incorrect estimates incur no additional gradient evaluations. In contrast, in Algorithm~\ref{alg:GDS_linesearch_naive}, each failed estimate of $L$ requires rerunning Subroutine \ref{alg:GDS_linesearch_subroutine} and hence incurs $T_k^s$ gradient evaluations of $\nabla f$.
Therefore, extra effort in the above analysis is needed to bound $\sum\nolimits_{k=1}^N\sum\nolimits_{s=1}^{S_k}T_k^s$.

A drawback of such naive linesearch strategy is that the initial guess $L_0$ needs to satisfy $L_0\ge \varepsilon/\|x_0 - x^*\|^2$, which still requires the prior knowledge of the distance to an optimal solution $x^*$.
Reexamining the proof of Corollary \ref{thm:GDS_naive_corollary} reveals that this drawback stems from the need to bound
the total cost $\sum_{s=1}^{S_k}T_k^s$ at the first iteration $k=1$, which is further bounded by $4M/L_0$.
The sole purpose of requiring $L_0\ge \varepsilon/\|x_0 - x^*\|^2$ is to guarantee that the incurred cost at iteration $k=1$ remains in the order of $\cO(M/\varepsilon)$.
This observation suggests that the drawback can be overcome by modifying the first iteration in the naive approach. Our proposed modification, called parameter-free GDS and described in Algorithm~\ref{alg:GDS_linesearch}, removes the dependence on the prior knowledge of $\|x_0 - x^*\|^2$  by modifying the first outer iteration in the naive approach. 
\begin{algorithm}[h]
\caption{\label{alg:GDS_linesearch} Parameter-free GDS}
\begin{algorithmic}[1]
\State In Algorithm \ref{alg:GDS_linesearch_naive}, modify the initialization and $k=1$ iteration as follows:
\State Initialization: $\xtil_0=x_0\in X$, $L_0\le L$
\State Set $k=1$, $s=1$, and
$x_1^{0,s} = x_0$
\State Compute
$x_1^{1,s} = \operatorname{argmin}_{x\in X} \langle \nabla g(\xtil_0) + \nabla f(x_1^{0,s}), x\rangle + \tfrac{\eta_1^s}{2}\|x-\xtil_0\|^2 + \tfrac{p_1^{1,s}}{2}\|x-x_1^{0,s}\|^2$,
where $\eta_1^s=L_1^s$, $p_1^{1,s}=M_1^{1,s}$,  $M_1^{1,s}=L_1^s = 2^{i_1^{1,s}}L_0$, and $i_1^{1,s}$ is the smallest nonnegative integer s.t. $g(x_1^{1,s}) \le g(\xtil_0) + \langle \nabla g(\xtil_0), x_1^{1,s}-\xtil_0\rangle + \tfrac{L_1^s}{2}\|x_1^{1,s}-\xtil_0\|^2$
and
$f(x_1^{1,s}) \le f(x_1^{0,s}) + \langle \nabla f(x_1^{0,s}), x_1^{1,s}-x_1^{0,s}\rangle + \tfrac{M_1^{1,s}}{2}\|x_1^{1,s} -x_1^{0,s}\|^2$\label{state:k1s1}
\State Set $x_1^s = \xtil_1^s = x_1^{1,s}$, $T_1^s = 1$, $M_1^s = M_1^{1,s}$ and $s=2$ 
\vspace{5pt}
\State  Compute
$L_1^s = L_1^1/2^{s-1}$,
$\eta_1^s = L_1^s$,
$M_1^{0,s} = \max\{M_1^{s-1}, \eta_1^s\}$\label{state:k1s>=2}
\State Call subroutine
$(x_1^s, \xtil_1^s, M_1^s)
=
\cA(x_0, \nabla g(\xtil_0), \eta_1^s, M_1^{0,s})$ in Algorithm~\ref{alg:GDS_linesearch_subroutine}
\State If
$g(\xtil_1^s) \le g(\xtil_0) + \langle \nabla g(\xtil_0), \xtil_1^s - \xtil_0\rangle + \tfrac{L_1^s}{2}\|\xtil_1^s - \xtil_0\|^2$,
 set $s\gets s+1$ and return to line \ref{state:k1s>=2}; otherwise, set
$S_1 = s-1$,
$L_1 = L_1^{S_1}$, 
$\xtil_1 = \xtil_1^{S_1}$, $x_1 = x_1^{S_1}$\label{state:outer_condition_for_GDS_linesearch}
\end{algorithmic}
\end{algorithm}

In a nutshell, Algorithm~\ref{alg:GDS_linesearch} modifies the way of estimating $L$ and $M$ in the first outer iteration as follows.
Starting from an arbitrary $L_0<L$, we first increase the estimates of  $L$ and $M$ until both the inner and outer smoothness conditions in line \ref{state:k1s1} are satisfied,
while maintaining $L_1^1=M_1^{1,1}$ to avoid the uncontrolled $\nabla f$ evaluation cost caused by an excessively small $L_0$.
We then decrease the estimate of $L$ until the outer smoothness condition in line \ref{state:outer_condition_for_GDS_linesearch} is violated for the first time. 
As a result, the accepted estimate satisfies $L_1\le 2L$, which ensures a controlled $\nabla g$ evaluation cost. 

As in the discussion following Algorithm~\ref{alg:GDS_linesearch_naive}, Theorem~\ref{thm:GDS_complexity} applies directly to Algorithm~\ref{alg:GDS_linesearch}. We are now ready to present the complexity of Algorithm~\ref{alg:GDS_linesearch}.
\begin{corollary}\label{thm:GDS_corollary}
     Algorithm \ref{alg:GDS_linesearch}
    with Subroutine~\ref{alg:GDS_linesearch_subroutine} requires at most $\lceil 2L\|x_0 - x^*\|^2/\varepsilon\rceil$  and $\lceil 12 M\|x_0 - x^*\|^2/\varepsilon\rceil$ evaluations of $\nabla g$ and $\nabla f$ to compute an $\varepsilon$-solution.
\end{corollary}
\begin{proof}
The search rules for $L_k^s$ and $M_k^{t,s}$ in Algorithm~\ref{alg:GDS_linesearch} and Subroutine~\ref{alg:GDS_linesearch_subroutine} ensure the integrality requirement of $M_k^{0,s}/\eta_k^s$ in Subroutine~\ref{alg:GDS_linesearch_subroutine} and the relations $\eta_k^s = L_k^s$, $\eta_k=L_k\le2L$ and $p_k^{t,s}=M_k^{t,s}\le 2M$. 
At the first outer iteration ($k=1$),  Algorithm~\ref{alg:GDS_linesearch} performs a decreasing search for an estimate of $L$ over $s\ge2$. During this decreasing search process, 
if the smoothness inequality in line~\ref{state:outer_condition_for_GDS_linesearch} of Algorithm~\ref{alg:GDS_linesearch} is tested and not violated at some trial value $L_1^{s_\varepsilon}$ such that $\varepsilon/(2\|x_0-x^*\|^2) < L_1^{s_\varepsilon}\le \varepsilon/\|x_0-x^*\|^2$, then  Theorem~\ref{thm:GDS_complexity} with Lemma~\ref{thm:GDS_linesearch_maintain_eq} indicates that $\xub_1 :=\xtil_1^{s_\varepsilon}$ is an $\varepsilon$-solution. In this case, the number of evaluations of $\nabla g$ and  $\nabla f$ is $1$ and 
\(\sum_{s=1}^{s_\varepsilon}T_1^s 
\le\sum_{s=1}^{s_\varepsilon}\sum_{t=1}^{T_1^s} 2M/M_1^{t,s} 
= 2M\sum_{s=1}^{s_\varepsilon}1/L_1^s 
\le 4M/L_1^{s_\varepsilon} 
< 8M\|x_0-x^*\|^2/\varepsilon\), respectively.
Otherwise, the decreasing search process accepts some $L_1 > \varepsilon/\|x_0-x^*\|^2$ and hence the number of evaluations of $\nabla f$ at the first outer iteration is $
\sum\nolimits_{s=1}^{S_1+1} T_1^s
         \le
         \sum\nolimits_{s=1}^{S_1+1} \sum\nolimits_{t=1}^{T_1^s} 2M/M_1^{t,s}
         \le
         {2M}/{L_1^{S_1+1}} + 2M\sum\nolimits_{s=1}^{S_1} {1}/{L_1^s} 
         \le
         {8M}/{L_1}
         < 8M\|x_0-x^*\|^2/\varepsilon
$.  Let $N_\varepsilon\ge2$ be the first outer iteration index such that $\sum_{k=1}^{N_\varepsilon} 1/L_k \ge \|x_0-x^*\|^2/\varepsilon$. Then, Theorem~\ref{thm:GDS_complexity} with Lemma~\ref{thm:GDS_linesearch_maintain_eq} indicates that $\xub_{N_\varepsilon}$ is an $\varepsilon$-solution and the required number of evaluations of $\nabla g$ is $N_\varepsilon \le \lceil 2L\|x_0-x^*\|^2/\varepsilon\rceil$. 
Moreover,  for $k\ge 2$,  \(\sum\nolimits_{s=1}^{S_k} T_k^s
        \le
        \sum\nolimits_{s=1}^{S_k} \sum\nolimits_{t=1}^{T_k^s}2M/M_k^{t,s}
        =    2M\sum\nolimits_{s=1}^{S_k} 1/L_k^s \le 4M/L_k^1 =  4M/L_{k-1}
        \). So the total number of evaluations of $\nabla f$ is 
$
\sum_{s=1}^{S_1+1}T_1^s 
+
\sum\nolimits_{k=2}^{N_\varepsilon}\sum\nolimits_{s=1}^{S_k}T_k^s 
< 
8M\|x_0-x^*\|^2/\varepsilon 
+ 
4M\sum\nolimits_{k=1}^{N_\varepsilon-1} 1/L_k
< 12M\|x_0-x^*\|^2/\varepsilon.
$
\end{proof}

Corollary~\ref{thm:GDS_corollary} shows that the proposed parameter-free GDS method achieves complexities of $\cO(L/\varepsilon)$ for $\nabla g$ and $\cO(M/\varepsilon)$ for $\nabla f$. It is also easy to verify that the total number of function evaluations of $g$ and $f$ are of the same order respectively. These complexity bounds match those obtained in Section~\ref{sec:GDS} for GDS with known constants $L$ and $M$. This raises several natural research questions. First, can the framework be extended from the Euclidean setting to general prox structures? Second, can accelerated methods be incorporated into the framework to obtain optimal complexities? Third, what modifications are needed when $\nabla f$ is only H\"older continuous rather than Lipschitz continuous? The next section addresses all these questions.

\section{Prototype: Universal Gradient Sliding (UGS)}
\label{sec:universal}
In this section, we develop a prototypical UGS method that is almost parameter-free for solving problem \eqref{eq:problem_of_interest},
where function $f$ satisfies the H\"older condition \eqref{eq:holder_descent}.
We incorporate a prox-function with respect to a general norm $\|\cdot\|$, accelerated sliding and  universal schemes, and attain the optimal oracle complexities for both $g$ and $f$.
The task is to design a double linesearch procedure without requiring prior knowledge of the problem constants $L$, $M_\nu$ or $\nu$.
We show that the proposed UGS
requires at most $\cO((M_\nu/\varepsilon)^{2/(1+3\nu)})$ and $\cO((L/\varepsilon)^{1/2})$ (sub)gradient evaluations of $f$ and $g$, respectively, under a single condition on the initial estimate $L_0$. 

The proposed UGS applies to problem \eqref{eq:problem_of_interest} under the assumption 
\begin{equation}\label{eq:LMnv_assumption}
    L \le \overline{M}_{\nu,\varepsilon}:=[\tfrac{1-\nu}{1+\nu}(\tfrac{1}{\varepsilon})]^{\tfrac{1-\nu}{1+\nu}} M_\nu^{\tfrac{2}{1+\nu}}.
\end{equation}
It captures the regime in which the $f$-component is more expensive in terms of oracle complexity, making it beneficial to skip evaluations of $\nabla g$. 
Here and throughout the paper we follow the convention that $0^0=1$ in \eqref{eq:LMnv_assumption}, so that $L\le \overline{M}_{\nu=1,\varepsilon} = M_{\nu=1}$. Similar convention concerning $0^0$ was used previously in the analysis of \cite{nesterov2015universal}.
\subsection{Proposed UGS algorithm}\label{sec:universal_base}
The proposed UGS adopts the same structure as in previous sections: at each iteration, Algorithm~\ref{alg:universal_linesearch} calls Subroutine~\ref{alg:universal_subroutine_linesearch} to approximately solve a subproblem.
To accommodate non-Euclidean settings, we use  a Bregman divergence as the proximity measure; see also~\cite{beck2003mirror}. Specifically, $V(\cdot,\cdot)$ is called a Bregman divergence on $X$ with modulus $\sigma_\omega>0$ (assumed $\sigma_\omega=1$ without loss of generality), if it is generated by a strongly convex function $\omega:X\to\mathbb{R}$ via
\begin{equation}\label{eq:Breg_definition}
    V(x,z):=\omega(z)-\omega(x)-\langle\nabla \omega(x),z-x\rangle
    \ge \tfrac{\sigma_\omega}{2}\|z-x\|^2,
    \qquad \forall x,z\in X.
\end{equation}
For notational convenience, in this paper we use the notation
$
    V^*:=V(x_0,x^*).
$

\begin{algorithm}[!h]
\caption{Main Procedure for UGS}
\label{alg:universal_linesearch}
\begin{algorithmic}[1]
\State Initialization: $x_0 = \xtil_0 = \xub_0 \in X$, $L_0$ s.t.
\( \varepsilon/V^*\le L_0\le L\) and  $M_0=L_0$
\For{$k=1,\dots, N$}
\State Set $s=1$ 
\State Set
$L_k^s = 2^{s-1}L_{k-1}$,
$\eta_k^s = L_k^s\gamma_k^s$, and 
$M_{k}^{0,s} = \max\{ M_0, L_k^s\}$\label{state:universal_outeralg_kge2}
\State
Compute 
$\xlb_k^s =  (1-\gamma_k^s)\xub_{k-1} + \gamma_k^s x_{k-1}
$, where 
\begin{equation}\label{eq:gamma_recur}
   \gamma_1^s=1, \quad \gamma_k^s >0,  \text{ s.t. }   L_k^s (\gamma_k^s)^2 = L_{k-1}(\gamma_{k-1}^{S_{k-1}})^2(1- \gamma_k^s), \forall k\ge 2
\end{equation}
\State Call subroutine $ (x_k^s, \tilde x_k^s, \xub_k^s)
=
\mathcal{A}(\nabla g(\xlb_k^s), \eta_k^s, M_{k}^{0,s})$ in Algorithm~\ref{alg:universal_subroutine_linesearch}
\State Set $S_k=s$,
$L_k=L_k^{S_k}$,
$x_k = x_k^{S_k}$,
$\xtil_k = \xtil_k^{S_k}$, $\xub_k=\xub_k^{S_k}$, if
\begin{align}\label{eq:universal_g_ineq_s}
    g(\xub_k^s)\le g(\xlb_k^s)+\langle \nabla g(\xlb_k^s), \xub_k^s-\xlb_k^s\rangle + \tfrac{L_k^s}{2}\|\xub_k^s-\xlb_k^s\|^2
\end{align}
Otherwise, set $s\gets s+1$ and return to line \ref{state:universal_outeralg_kge2}  \label{state:universal_mainalg_indexing}
\EndFor
\State Output
$\xub_N$
\end{algorithmic}
\end{algorithm}
\begin{algorithm}[!h]
    \caption{linesearch subroutine $\cA$}
    \label{alg:universal_subroutine_linesearch}
    \begin{algorithmic}[1]
        \State Input: $\nabla g(\xlb_k^s)$, $\eta_k^s$, $M_{k}^{0,s}$
        \State Set $x_k^{0,s} = x_{k-1}$, $\xtil_k^{0,s} = \xtil_{k-1}$, $c_k^{0,s}=1$, $t=1$
        \State Compute
        \vspace{-8pt}
        \begin{equation}
        \begin{split}
            \xlb_k^{t,s} = & (1-\gamma_k^s)\xub_{k-1}+\gamma_k^s[(1-\alpha_k^{t,s})\xtil_k^{t-1,s}+\alpha_k^{t,s}x_k^{t-1,s}]
            \\
            x_k^{t,s} = & \operatorname{argmin}_{x\in X} \langle \nabla g(\xlb_k^s) + f'(\xlb_k^{t,s}), x\rangle + \eta_k^sV(x_{k-1},x)+p_k^{t,s}V(x_k^{t-1,s},x)\label{eq:universal_subroutine_update}
            \\
            \xtil_k^{t,s} = & (1-\alpha_k^{t,s})\xtil_k^{t-1,s}+\alpha_k^{t,s}x_k^{t,s}
            \\
            \xub_k^{t,s} = & (1-\gamma_k^s)\xub_{k-1}+\gamma_k^s\xtil_k^{t,s}
            \end{split}
        \end{equation}
        Here the parameters are set as follows:
        \vspace{-5pt}
    \begin{equation}\label{eq:alpha_recur}
        \alpha_k^{1,s} =1, \quad \alpha_k^{t,s} >0, 
        \text{ s.t. } M_k^{t,s}(\alpha_k^{t,s})^2 = M_k^{t-1,s}(\alpha_k^{t-1,s})^2(1-\alpha_k^{t,s}), \forall t\ge2
    \end{equation}
    $
    p_k^{t,s} = L_k^s \gamma_k^{s}(1-\alpha_k^{t,s})/\alpha_k^{t,s}
    + \gamma_k^sM_k^{t,s}\alpha_k^{t,s}
    $, and $M_{k}^{t,s}=2^{i_k^{t,s}}c_k^{t-1,s}M_{k}^{t-1,s}$, where $i_k^{t,s}$ is the smallest nonnegative integer such that the following holds:
    \vspace{-5pt}
            \begin{equation}
            \label{eq:universal_f_ineq_s}
            f(\xub_k^{t,s})\le f(\xlb_k^{t,s})+
    \langle f'(\xlb_k^{t,s}), \xub_k^{t,s} - \xlb_k^{t,s}\rangle + \tfrac{M_k^{t,s}}{2}\|\xub_k^{t,s}-\xlb_k^{t,s}\|^2+\tfrac{\gamma_k^s\alpha_k^{t,s}\varepsilon}{2}
\end{equation}\label{state:universal_subroutine_M_linesearch}
    \vspace{-10pt}
    \State\label{state:universal_T_test1} If {$M_k^{t,s}(\alpha_k^{t,s})^2=L_k^s$}, set $T_k^s=t$, \textbf{terminate} and output 
                $x_k^s=x_k^{T_k^s,s}$,
        $\tilde x_k^s = \tilde x_k^{T_k^s,s}$,
        $\xub_k^s = \xub_k^{T_k^s,s}$;
        else if
        {$M_k^{t,s} [(\alpha_k^{t,s})^4+2(\alpha_k^{t,s})^2-(\alpha_k^{t,s})^3\sqrt{(\alpha_k^{t,s})^2+4}]/2 \le L_k^s$},
        set $c_k^{t,s}=L_k^s(\alpha_k^{t,s})^4M_k^{t,s}/[M_k^{t,s}(\alpha_k^{t,s})^2 - L_k^s]^2$; otherwise, set $c_k^{t,s}=1$
        \vspace{5pt}
        \State Set $t\gets t+1$ and return to line \ref{state:universal_subroutine_M_linesearch}
    \end{algorithmic}
\end{algorithm}
To achieve the optimal gradient complexity for the smooth component $g$, Algorithm~\ref{alg:universal_linesearch} incorporates the accelerated gradient method in \cite{nesterov1983method,lan2020first}. Meanwhile, Subroutine~\ref{alg:universal_subroutine_linesearch} combines two ideas from the literature. First, following the universal gradient method~\cite{nesterov2015universal}, it treats the nonsmooth or weakly smooth component $f$ as approximately smooth and applies an accelerated scheme; this is reflected in the linesearch condition \eqref{eq:universal_f_ineq_s}, which permits the error term $\gamma_k^s\alpha_k^{t,s}\varepsilon/2$. Second, following the accelerated gradient sliding method~\cite{lan2022accelerated}, the inner gradient evaluation point $\xlb_k^{t,s}$ is chosen as a convex combination of $\{x_k^{\tau,s}\}_{\tau< t}$ and the outer point $\xub_{k-1}$. This construction is crucial for embedding the doubly accelerated structure into the sliding framework and obtaining the optimal complexity bounds. Note that \cite{lan2022accelerated} considers only a sum of two smooth functions with known Lipschitz constants, while we propose a parameter-free algorithm and extend the structure to the case where the objective consists of a smooth function and a function with H\"older continuous gradient.

The proposed UGS also introduces two new ingredients that require  additional techniques. First, in GDS, the quantity $M_k^{t,s}$ in Algorithm~\ref{alg:GDS_linesearch_subroutine} serves both as the linesearch estimate of the unknown constant $M$ and as the parameter used to verify \eqref{eq:GDS_f_ineq}. As a result, once $M_k^{t,s}$ is sufficiently large, \eqref{eq:GDS_f_ineq} holds globally. For a general function $f$ satisfying the $(M_\nu,\nu)$-H\"older condition in \eqref{eq:holder_descent}, however, the validity of the linesearch must first be related to the pair $(M_\nu,\nu)$. By adapting the strategy of \cite{nesterov2015universal}, we obtain the following proposition; its proof is deferred to the appendix.
\begin{proposition}\label{thm:universal_estimates_bound}
Assume that the relation between $L$ and $M_\nu$ in \eqref{eq:LMnv_assumption} holds.
For any $k\ge 1$ and $s\ge 1$, in Algorithm~\ref{alg:universal_linesearch} and Subroutine~\ref{alg:universal_subroutine_linesearch}, the estimates $L_k$ and $M_k^{t,s}$ of Lipschitz and H\"older constants satisfy the following relations:
{\small \begin{equation}
\label{eq:LM_within_2true}
\begin{aligned}
    & L_k\le 2L, M_k^{1,s}\le 2\overline{M}_{\nu,\varepsilon}{\gamma_k^s}^{ -\tfrac{1-\nu}{1+\nu}}
    , 
    M_k^{t,s}\le 2\overline{M}_{\nu,\varepsilon}{(\gamma_k^s\alpha_k^{t,s})}^{ -\tfrac{1-\nu}{1+\nu}}
    , \forall 2\le t\le T_k^s-1.
\end{aligned}
\end{equation}}
\end{proposition}

Second, as in Algorithm~\ref{alg:GDS_linesearch_subroutine}, the number of inner iterations $T_k^s$ in Subroutine~\ref{alg:universal_subroutine_linesearch} must be chosen adaptively to enforce a key balancing condition. Due to the accelerated structure of Subroutine~\ref{alg:universal_subroutine_linesearch}, however, this task is more delicate, and the strategy used in Subroutine~\ref{alg:GDS_linesearch_subroutine} is no longer applicable. 
We therefore introduce a new mechanism involving the coefficient
$c_k^{t,s}$ to enforce the terminal balancing condition
\(
    M_k^{T_k^s,s}(\alpha_k^{T_k^s,s})^2=L_k^s .
\)
Specifically, the procedure starts from
$M_k^{1,s}(\alpha_k^{1,s})^2\ge L_k^s$, decreases
$M_k^{t,s}(\alpha_k^{t,s})^2$ over the inner iterations, and adjusts
$c_k^{t,s}$ when needed so that the balancing equality holds exactly at the final step.
As shown in Proposition~\ref{thm:Ak1s_ge_Lks}, this balancing relation links the number of inner iterations
to the local estimates of the problem constants and is essential for achieving the desired complexity bounds.

\subsection{Convergence analysis of UGS}\label{sec:universal_convergence_properties}
This subsection establishes the convergence properties of UGS. Throughout the rest of the paper, for any $k,s,t\ge1$, define
\begin{align}
\label{eq:AGamma}
A_k^{t,s} := M_k^{t,s}(\alpha_k^{t,s})^2,\
\Gamma_k^{s} := L_k^{s}(\gamma_k^{s})^2,\
\xltil_k^{t,s}
:= (1-\alpha_k^{t,s})\,\xtil_k^{t-1,s}
+ \alpha_k^{t,s} x_k^{t-1,s}.
\end{align}
Noting 
\eqref{eq:alpha_recur}
and \eqref{eq:gamma_recur}, we have, respectively, 
\begin{align}
    \label{eq:AGamma_recur}
    A_k^{t,s} = A_k^{t-1,s}(1-\alpha_k^{t,s}),\ \forall t\ge 2, k,s\ge1,
    \;
    \Gamma_k^s = \Gamma_{k-1}^{S_{k-1}}(1-\gamma_k^s),\ \forall k\ge 2, s\ge1.
\end{align}

The following two lemmas will be used for convergence analysis. The proof of Lemma \ref{thm:recursion_property} is deferred to the appendix, and Lemma \ref{thm:three_point_Breg} is well-known (see, e.g., \cite{lan2020first}).
\begin{lemma}\label{thm:recursion_property}
    Let $\{E_\tau\}_{\tau\ge1}$ be a nondecreasing sequence of positive scalars.
    Set $\lambda_1=1$,
    and for $\tau\ge2$, let $\lambda_\tau\ge0$ satisfy $\Lambda_\tau = \Lambda_{\tau-1}(1-\lambda_\tau)$,
    where $\Lambda_\tau = E_\tau(\lambda_\tau)^2$, $\forall \tau\ge1$. Then, the following statements hold:
   \begin{enumerate}[label=(\alph*)]
       \item For all $\tau \ge 2$, we have $\lambda_\tau \in (0,1)$ and hence $0<\Lambda_\tau < \Lambda_{\tau-1}$. \label{item:lambda_in_01_Lambda_decrease}
       \item Let sequence $\{b_\tau\}_{\tau\ge0}$ and $\{d_\tau\}_{\tau\ge1}$ satisfy $b_\tau = (1-\lambda_\tau)b_{\tau-1}+\lambda_\tau d_\tau$, for all $\tau\ge1$. Then, for any $\tau\ge1$, it holds that $b_\tau= \Lambda_\tau \sum_{i=1}^\tau \lambda_i d_i/\Lambda_i$. Specifically, when $b_\tau=1$, for $\tau\ge0$ and $d_\tau=1$ for all $\tau\ge1$, we have $1=\Lambda_\tau \sum_{i=1}^\tau \lambda_i /\Lambda_i$.\label{item:AT_eq_sum_a_over_A}
       \item For $\tau\ge2$, $\lambda_\tau$ is increasing in $\Lambda_{\tau-1}$ and decreasing in $E_\tau$. Consequently, for fixed $\Lambda_{\tau-1}$, the quantity $\Lambda_\tau=\Lambda_{\tau-1}(1-\lambda_\tau)$ is increasing in $E_\tau$.\label{item:lambda_Lambda_E}
       \item $\lambda_\tau$ strictly decreases in $\tau$.\label{item:lambda_decrease}
        \item For all $\tau\ge1$, $\lambda_\tau\le2/(\tau+1)$.\label{item:lambda_bound}
       \item Fix any  $E_\text{fix}$,
       such that $\Lambda_{\tau}\le E_\text{fix} < \Lambda_{\tau-1}$, for some $\tau\ge2$. Let  $c = E_\text{fix}/[(1-E_\text{fix}/\Lambda_{\tau-1})^2 E_\tau]$, and let $\hat{\lambda}_\tau>0$ solve $cE_\tau (\hat{\lambda}_\tau)^2 = \Lambda_{\tau-1}(1-\hat{\lambda}_\tau)$. Denote $\hat{\Lambda}_\tau := c E_\tau \hat{\lambda}^2$. Then $c\ge1$ and $\hat{\Lambda}_\tau = E_{\mathrm{fix}}$.\label{item:c_def}
       \item For any \(\tau\ge2\), let \(\overline{E}_\tau=qE_\tau\) for some \(q\ge1\), and let \(\overline{\lambda}_\tau>0\) satisfy
\(
\overline{E}_\tau(\overline{\lambda}_\tau)^2=\Lambda_{\tau-1}(1-\overline{\lambda}_\tau).
\)
Then,
\(
\lambda_\tau\le \sqrt{q}\overline{\lambda}_\tau.
\)
\label{item:lambda_compare_L_increase}
\item For any \(\tau\ge1\), if \(E_{\tau+1}=E_\tau\), then
\(
\lambda_{\tau+1}\ge \tfrac{2}{1+\sqrt{5}}\lambda_\tau.
\)
\label{item:lambda_same_E_ratio}
\item For all $\tau\ge1$, we have $\Lambda_\tau\ge E_1/\tau^2$.
\label{item:Gamma_k_ge_L1_over_ksq}
   \end{enumerate}
\end{lemma}

\begin{lemma}[\cite{lan2020first} Lemma~3.5]\label{thm:three_point_Breg}
    Let the convex function $q:X\to\mathbb{R}$, the points $z,z'\in X$ and the scalars $\mu_1,\mu_2\ge0$ be given. Let $V(\cdot,\cdot)$ be defined in \eqref{eq:Breg_definition}. If $u^*\in\operatorname{argmin}_{u\in X} q(u)+ \mu_1 V(z,u)+\mu_2V(z',u)$,
    then for any $u\in X$, we have bound
    $
        q(u^*) + \mu_1V(z,u^*)+\mu_2V(z',u^*)
        \le
        q(u) + \mu_1V(z,u)+\mu_2V(z',u) -(\mu_1+\mu_2)V(u^*,u).
    $
\end{lemma}

Proposition~\ref{thm:universal_subroutine_convergence} establishes the convergence of
Subroutine~\ref{alg:universal_subroutine_linesearch}; the resulting bound depends implicitly on
the number $T_k^s$ of inner iterations.
\begin{proposition}\label{thm:universal_subroutine_convergence}
    Fix any $x\in X$, $k\ge1$ and $s\ge1 $. At the termination of Subroutine~\ref{alg:universal_subroutine_linesearch}, we have 
\begin{equation*}
\begin{aligned}
        &f(\xub_k^s) - f(z_k^s) + \langle \nabla g(\xlb_k^s), \xub_k^s-z_k^s\rangle-\tfrac{\gamma_k^s\varepsilon}{2}  +\tfrac{\gamma_k^s\eta_k^s}{2}\|\xtil_k^s-x_{k-1}\|^2
    \\
    \le&
    \gamma_k^s(\eta_k^s+A_k^{T_k^s,s}p_k^{1,s}/A_k^{1,s})V(x_{k-1},x)
    -
    \gamma_k^s\alpha_k^{T_k^s,s}(p_k^{T_k^s,s}+\eta_k^s)V(x_k^s,x),
    \end{aligned}
    \end{equation*}
    where $z_k^s
:= (1-\gamma_k^s)\,\xub_{k-1}
+ \gamma_k^s x$.
\end{proposition}

\begin{proof}
Fix any $t\ge1$. By the definitions of $\xub_k^{t,s}$, $\xlb_k^{t,s}$, $\xtil_k^{t,s}$, $\xltil_k^{t,s}$ and $z_k^s$, we observe that 
$\xub_k^{t,s}-\xlb_k^{t,s}
=
\gamma_k^s\alpha_k^{t,s}(x_k^{t,s}-x_k^{t-1,s})$,
and
$\xub_k^{t,s}-(1-\alpha_k^{t,s})\xub_k^{t-1,s}-\alpha_k^{t,s} z_k^s
=
\gamma_k^s\alpha_k^{t,s}(x_k^{t,s} - x)$.
Utilizing \eqref{eq:universal_f_ineq_s}, convexity of $f(\cdot)$ and the above observations yields
\begin{equation*}\label{eq:universal_f_expend}
\begin{aligned}
&f(\xub_k^{t,s}) - (1-\alpha_k^{t,s})f(\xub_k^{t-1,s}) - \alpha_k^{t,s} f(z_k^s) + \langle \nabla g(\xlb_k^s), \xub_k^{t,s}-(1-\alpha_k^{t,s})\xub_k^{t-1,s}-\alpha_k^{t,s} z_k^s \rangle
\\
\le&
f(\xlb_k^{t,s}) +\langle f'(\xlb_k^{t,s}) , \xub_k^{t,s}-\xlb_k^{t,s}\rangle +\tfrac{M_k^{t,s}}{2}\|\xub_k^{t,s}-\xlb_k^{t,s}\|^2 + \tfrac{\gamma_k^s\alpha_k^{t,s}\varepsilon}{2}
\\
&- (1-\alpha_k^{t,s}) [f(\xlb_k^{t,s}) + \langle f'(\xlb_k^{t,s}), \xub_k^{t-1,s}-\xlb_k^{t,s}\rangle]
-\alpha_k^{t,s}[f(\xlb_k^{t,s}) + \langle f'(\xlb_k^{t,s}), z_k^s-\xlb_k^{t,s}\rangle]
\\
&+ \langle \nabla g(\xlb_k^s), \xub_k^{t,s}-(1-\alpha_k^{t,s})\xub_k^{t-1,s}-\alpha_k^{t,s} z_k^s \rangle
\\
\le&
 \gamma_k^s\alpha_k^{t,s}\langle f'(\xlb_k^{t,s})+\nabla g(\xlb_k^s), x_k^{t,s} - x \rangle
+  \tfrac{M_k^{t,s}(\gamma_k^s\alpha_k^{t,s})^2}{2}\|x_k^{t,s}-x_k^{t-1,s}\|^2+\tfrac{\gamma_k^s\alpha_k^{t,s}\varepsilon}{2}.
\end{aligned}
\end{equation*}
Here we make two observations. First, Lemma~\ref{thm:three_point_Breg} applied to \eqref{eq:universal_subroutine_update} yields
\begin{equation*}
\begin{aligned}
&\langle
f'(\xlb_k^{t,s})+\nabla g(\xlb_k^s),
x_k^{t,s} - x
 \rangle
 -
 p_k^{t,s} (V(x_k^{t-1,s},x) -V(x_k^{t-1,s},x_k^{t,s})- V(x_k^{t,s},x) )
\\
\le\ &
\eta_k^s (V(x_{k-1},x)-V(x_{k-1},x_k^{t,s}) - V(x_k^{t,s},x) ).
\end{aligned}
\end{equation*}
Second, by the definition of $p_k^{t,s}$ in Subroutine~\ref{alg:universal_subroutine_linesearch}
and Lemma~\ref{thm:recursion_property}\ref{item:lambda_in_01_Lambda_decrease},
we have $p_k^{t,s}\ge M_k^{t,s}\gamma_k^s\alpha_k^{t,s}$. Thus, by
\eqref{eq:Breg_definition},
\(
 {M_k^{t,s}(\gamma_k^s\alpha_k^{t,s})^2}\|x_k^{t,s}-x_k^{t-1,s}\|^2/2\le p_k^{t,s}\gamma_k^s\alpha_k^{t,s}V(x_k^{t-1,s}, x_k^{t,s}).   
\)
Applying the two observations we have
{\small 
\begin{equation}\label{eq:universal_single_f}
\begin{aligned}
&f(\xub_k^{t,s}) - (1-\alpha_k^{t,s})f(\xub_k^{t-1,s}) - \alpha_k^{t,s} f(z_k^s)
+  \langle \nabla g(\xlb_k^s), \xub_k^{t,s}-(1-\alpha_k^{t,s})\xub_k^{t-1,s}-\alpha_k^{t,s} z_k^s \rangle
\\
\le&
\gamma_k^s \alpha_k^{t,s}
 [
\eta_k^s (V(x_{k-1},x) - V(x_{k-1},x_k^{t,s}) )
+p_k^{t,s}V(x_k^{t-1,s},x)
- (p_k^{t,s}+\eta_k^s)V(x_k^{t,s},x)
 ]
+ \tfrac{\gamma_k^s\alpha_k^{t,s}\varepsilon}{2}.
\end{aligned}
\end{equation}
}
Noting the definition of $A_k^{t,s}$ in \eqref{eq:AGamma} and its recursion \eqref{eq:AGamma_recur}, given $p_k^{t,s}$ in Subroutine \ref{alg:universal_subroutine_linesearch} and $\eta_k^s$
in Algorithm~\ref{alg:universal_linesearch}, we have
\(
 \alpha_k^{t,s}p_k^{t,s}/A_k^{t,s}
=
 \alpha_k^{t-1,s}(p_k^{t-1,s}+\eta_k^s)/A_k^{t-1,s}
\)
for all $t\ge 2$. 
Hence,  multiplying both sides of \eqref{eq:universal_single_f} by $A_k^{T_k^s,s}/A_k^{t,s}$, summing over $t=1,\dots,T_k^s$, and using $\alpha_k^{1,s}=1$, we have
\begin{align*}
&f(\xub_k^{T_k^s,s}) - f(z_k^s) + \langle \nabla g(\xlb_k^s), \xub_k^{T_k^s,s}-z_k^s\rangle +\gamma_k^s\alpha_k^{T_k^s,s}(p_k^{T_k^s,s}+\eta_k^s)V(x_k^{T_k^s,s},x)
\\
\le & A_k^{T_k^s,s}\sum\nolimits_{t=1}^{T_k^s}\tfrac{\gamma_k^s \alpha_k^{t,s}}{A_k^{t,s}}
 [\eta_k^s(V(x_{k-1},x) - V(x_{k-1},x_k^{t,s}) )+ \tfrac{\varepsilon}{2}]
 + \tfrac{\gamma_k^s A_k^{T_k^s,s}p_k^{1,s}}{A_k^{1,s}}V(x_k^{0,s},x)
\\
\le&
\gamma_k^s\eta_k^s (V(x_{k-1},x)- \tfrac{1}{2}\|\xtil_k^{T_k^s,s}-x_{k-1}\|^2 )
+ \tfrac{\gamma_k^s A_k^{T_k^s,s}p_k^{1,s}}{A_k^{1,s}}V(x_k^{0,s},x)
+ \tfrac{\gamma_k^s\varepsilon}{2}.
\end{align*}
The above derivation also uses the following facts with the help of
Lemma~\ref{thm:recursion_property}~\ref{item:AT_eq_sum_a_over_A}:
(i) \(
A_k^{T_k^s,s}\sum_{t=1}^{T_k^s} \alpha_k^{t,s}/A_k^{t,s} = 1
\), (ii)
\(
\xtil_k^{T_k^s,s}=A_k^{T_k^s,s}\sum_{t=1}^{T_k^s} \alpha_k^{t,s} x_k^{t,s}/A_k^{t,s}
\), (iii)
\(
 \|\xtil_k^{T_k^s,s}-x_{k-1}\|^2/A_k^{T_k^s,s}
\le
\sum\nolimits_{t=1}^{T_k^s} \alpha_k^{t,s}\|x_k^{t,s}-x_{k-1}\|^2/A_k^{t,s}
\le
2\sum\nolimits_{t=1}^{T_k^s} \alpha_k^{t,s}V(x_{k-1},x_k^{t,s})/A_k^{t,s}.
\)
Then, 
\(
\xub_k^s = \xub_k^{T_k^s,s},
\xtil_k^s = \xtil_k^{T_k^s,s},
x_k^s = x_k^{T_k^s,s},
x_{k-1}=x_k^{0,s}
\)
in subroutine
yield the conclusion.
 \end{proof}

Now we establish the convergence property of Algorithm~\ref{alg:universal_linesearch} in Proposition~\ref{thm:universal_outer_convergence}.
\begin{proposition}\label{thm:universal_outer_convergence}
    Fix any $x\in X$. Running $N$ iterations in Algorithm~\ref{alg:universal_linesearch} yields 
    \vspace{-10pt}
{\small \begin{equation*}
    \begin{aligned}
        & ( \Gamma_N^{S_N})^{-1} \left[ f(\xub_N) + g(\xub_N)-f(x)-g(x)-\tfrac{\varepsilon}{2} \right] 
         \\
         \le&
        \sum\nolimits_ {k=1}^N
         [
        \tfrac{\gamma_k^{S_k}}{\Gamma_k^{S_k}}(\eta_k^{S_k}+\tfrac{A_k^{T_k^{S_k},S_k}p_k^{1,S_k}}{A_k^{1,S_k}})
        V(x_{k-1},x)
        -
        \tfrac{\gamma_k^{S_k}\alpha_k^{T_k^{S_k},{S_k}}(p_k^{T_k^{S_k},{S_k}}+\eta_k^{S_k})}{\Gamma_k^{S_k}}V(x_k,x)
        ] .
    \end{aligned}
\end{equation*} }
\end{proposition}
\begin{proof}
Fix any $k\ge1$. From the description of Algorithm~\ref{alg:universal_linesearch} we have
\begin{align}\label{eq:universal_g_ineq}
         g(\xub_k) \le g(\xlb_k^{S_k}) + \langle \nabla g(\xlb_k^{S_k}),\xub_k-\xlb_k^{S_k}\rangle + \tfrac{L_k^{S_k}}{2}\|\xub_k-\xlb_k^{S_k}\|^2.
\end{align}
By definitions of $\xub_k$, $\xtil_k$ and $\xlb_k^{S_k}$, we observe
$
\xub_k-\xlb_k^{S_k} =
\gamma_k^{S_k}(\xtil_k-x_{k-1}).
$
Applying such observation, 
 inequalities \eqref{eq:Breg_definition} 
and \eqref{eq:universal_g_ineq}, the convexity of $g(\cdot)$, and the parameter choice $\eta_k^s = L_k^s\gamma_k^s$ in Algorithm~\ref{alg:universal_linesearch}, we have
{\small 
\begin{equation}\label{eq:universal_g_sum_ineq}
\begin{aligned}
    &g(\xub_k) - (1-\gamma_k^{S_k})g(\xub_{k-1})-\gamma_k^{S_k}g(x)
    \\
    \le & g(\xlb_k^{S_k}) + \langle \nabla g(\xlb_k^{S_k}), \xub_k-\xlb_k^{S_k}\rangle +  \tfrac{L_k^{S_k}(\gamma_k^{S_k})^2}{2}\|\xtil_k-x_{k-1}\|^2
    \\
    -& (1-\gamma_k^{S_k})( g(\xlb_k^{S_k}) + \langle \nabla g(\xlb_k^{S_k}), \xub_{k-1}-\xlb_k^{S_k}\rangle )
    - \gamma_k^{S_k} (g(\xlb_k^{S_k}) + \langle \nabla g(\xlb_k^{S_k}), x-\xlb_k^{S_k}\rangle)
    \\
    \le & \langle \nabla g(\xlb_k^{S_k}), \xub_k-(1-\gamma_k^{S_k})\xub_{k-1} - \gamma_k^{S_k} x\rangle +  \tfrac{\gamma_k^{S_k}\eta_k^{S_k}}{2}\|\xtil_k-x_{k-1}\|^2
    .
\end{aligned}
\end{equation}
}
Also, from convexity of $f$, Proposition~\ref{thm:universal_subroutine_convergence} (with $s=S_k$) and the relations $\xub_k=\xub_k^{S_k}$, $x_k=x_k^{S_k}$ and  $\xtil_k=\xtil_k^{S_k}$ introduced in Algorithm~\ref{alg:universal_linesearch}, we have
{\small
\begin{equation}
\begin{aligned}
    & f(\xub_k) - (1-\gamma_k^{S_k})f(\xub_{k-1}) - \gamma_k^{S_k} f(x) + \langle \nabla g(\xlb_k^{S_k}), \xub_k-(1-\gamma_k^{S_k})\xub_{k-1} - \gamma_k^{S_k} x\rangle
    \\
       \le &f(\xub_k) - f((1-\gamma_k^{S_k})\xub_{k-1} + \gamma_k^{S_k} x) + \langle \nabla g(\xlb_k^{S_k}), \xub_k-(1-\gamma_k^{S_k})\xub_{k-1} - \gamma_k^{S_k} x\rangle
    \\
    \le&
    \gamma_k^{S_k}(\eta_k^{S_k}+A_k^{T_k^{S_k},{S_k}}p_k^{1,{S_k}}/A_k^{1,{S_k}})V(x_{k-1},x)
    -
    \gamma_k^{S_k}\alpha_k^{T_k^{S_k},{S_k}}(p_k^{T_k^{S_k},{S_k}}+\eta_k^{S_k})V(x_k,x)
    \\
    & -\tfrac{\gamma_k^{S_k}\eta_k^{S_k}}{2}\|\xtil_k-x_{k-1}\|^2 + \tfrac{\gamma_k^{S_k}\varepsilon}{2}.
    \end{aligned}
    \end{equation}
}
Summing the above two results we obtain that
\begin{align*}
\begin{aligned}
    &f(\xub_k) + g(\xub_k)-f(x)-g(x)-(1-\gamma_k^{S_k})(f(\xub_{k-1}) + g(\xub_{k-1})-f(x)-g(x))-\tfrac{\gamma_k^{S_k}\varepsilon}{2}
    \\
    \le&
    \gamma_k^{S_k}(\eta_k^{S_k}+A_k^{T_k^{S_k},S_k}p_k^{1,S_k}/A_k^{1,S_k})V(x_{k-1},x)
    -
    \gamma_k^{S_k}\alpha_k^{T_k^{S_k},{S_k}}(p_k^{T_k^{S_k},{S_k}}+\eta_k^{S_k})V(x_k,x).
\end{aligned}
\end{align*}
The conclusion then follows from the recursion of $\Gamma_k^s$ in \eqref{eq:AGamma_recur}
and Lemma~\ref{thm:recursion_property}~\ref{item:AT_eq_sum_a_over_A}.
\end{proof}

The preceding proposition yields the following property of
Algorithm~\ref{alg:universal_linesearch}.
\begin{theorem}\label{thm:universal_parameters_and_convergence}
Fix any $x\in X$. Running $N$ iterations in Algorithm~\ref{alg:universal_linesearch} yields 
$
 f(\xub_N) + g(\xub_N)-f(x)-g(x)
         \le
         \tfrac{\varepsilon}{2}+
         \Gamma_N^{S_N}\sum\nolimits_ {k=1}^N
        (1+\tfrac{A_k^{T_k^{S_k},S_k}}{L_k^{S_k}})
        (V(x_{k-1},x)
        -
        V(x_k,x)).
$
Specifically, if $T_k^{S_k}$ is chosen so that $A_k^{T_k^{S_k},S_k}/ L_k^{S_k}=1$, for all $k\ge1$, then
\begin{align}\label{eq:universal_final_bound}
    f(\xub_N) + g(\xub_N)-f(x)-g(x)
         \le
         \tfrac{\varepsilon}{2}+
         2\Gamma_N^{S_N}
        V(x_{0},x).
\end{align}
\end{theorem}
\begin{proof}
The conclusions follow directly from Proposition~\ref{thm:universal_outer_convergence}
and the definitions of $\Gamma_k^s$ and $A_k^{t,s}$ in \eqref{eq:AGamma}, $\eta_k^s$
in Algorithm~\ref{alg:universal_linesearch}, and $p_k^{t,s}$ with $\alpha_k^{1,s}=1$
in Subroutine~\ref{alg:universal_subroutine_linesearch}.
\end{proof}

To derive explicit complexity from Theorem~\ref{thm:universal_parameters_and_convergence}, it remains to bound $\Gamma_N^{S_N}$ and the number of evaluations of $\nabla g$ and $f'$, which is the goal of the next subsection. Before proceeding, we highlight a key feature of Subroutine~\ref{alg:universal_subroutine_linesearch}. The result in
\eqref{eq:universal_final_bound} relies critically on the terminal identity
$A_k^{T_k^s,s}/L_k^s=1$. By
\eqref{eq:AGamma}, this is equivalent to choosing $T_k^s$ so that
\(
    M_k^{T_k^s,s}(\alpha_k^{T_k^s,s})^2=L_k^s.
\)
As discussed at the end of
Section~\ref{sec:GDS}, enforcing such terminal balance without knowing the problem
constants is the main difficulty in designing parameter-free GS methods. The
next proposition shows that
Subroutine~\ref{alg:universal_subroutine_linesearch} enforces this identity
automatically at termination.
\begin{proposition}\label{thm:Ak1s_ge_Lks}
Fix any $k \ge 1$ and $s \ge 1$. For any iteration  $t$ in Subroutine~\ref{alg:universal_subroutine_linesearch}, we have
\begin{align}
    \label{eq:universal_A_bound}
    L_k^s\le A_k^{t,s}\le
   (\tfrac{2}{t} )^{\tfrac{1+3\nu}{1+\nu}}
    2\overline{M}_{\nu,\varepsilon}
    (\gamma_k^s)^{-\tfrac{1-\nu}{1+\nu}}.
\end{align}
Moreover, Subroutine \ref{alg:universal_subroutine_linesearch} terminates after a finite number of iterations. Upon termination at the $T_k^s$-th iteration, we have $A_k^{T_k^s,s}=L_k^s$.
\end{proposition}
\begin{proof}
We first prove $L_k^s\le A_k^{t,s}$, $\forall t\ge1$ by induction. For $t=1$, we have $A_k^{1,s} = M_k^{1,s}(\alpha_k^{1,s})^2 = M_k^{1,s}\ge M_k^{0,s} \ge L_k^s$,
where the first inequality is guaranteed by the linesearch procedure in Subroutine~\ref{alg:universal_subroutine_linesearch} and the second by the definition of $M_k^{0,s}$ in Algorithm~\ref{alg:universal_linesearch}.
Suppose that  Subroutine~\ref{alg:universal_subroutine_linesearch} has not terminated at
iteration \(t\), it checks whether the smallest possible value of \(A_k^{t+1,s}\) (by Lemma~\ref{thm:recursion_property}\ref{item:lambda_Lambda_E}) at the
next inner iteration would be no larger than \(L_k^s\). This check is
implemented through the criterion
$M_k^{t,s}(\tilde\alpha_k^{t,s})^2\le L_k^{s}$, in which it is easy to verify that $(\tilde \alpha_k^{t,s})^2 := [(\alpha_k^{t,s})^4+2(\alpha_k^{t,s})^2-(\alpha_k^{t,s})^3\sqrt{(\alpha_k^{t,s})^2+4}]/2$ is the square of the positive root to the quadratic equation $(\tilde\alpha_k^{t,s})^2 = (\alpha_k^{t,s})^2(1-\tilde\alpha_k^{t,s})$.
Third, whenever the criterion
\(
    M_k^{t,s}(\tilde\alpha_k^{t,s})^2\le L_k^s
\)
is satisfied, the subroutine introduces a correction factor \(c_k^{t,s}\).
By Lemma~\ref{thm:recursion_property}\ref{item:c_def}, such factor $c_k^{t,s}\ge1$ makes the first 
trial value of \(A_k^{t+1,s}\) equal to \(L_k^s\).
Fourth, by Lemma~\ref{thm:recursion_property}\ref{item:lambda_Lambda_E}, the possible increase of \(M_k^{t+1,s}\) due to backtracking can only
increase the corresponding value of \(A_k^{t+1,s}\), which implies \(A_k^{t+1,s}\ge L_k^s\) .

Next, we prove the upper bound of $A_k^{t,s}$ in \eqref{eq:universal_A_bound}. By Lemma~\ref{thm:recursion_property}\ref{item:lambda_in_01_Lambda_decrease}, we have
\(
0<A_k^{t,s}< A_k^{t-1,s}, \forall\, t\ge2,
\)
which leads to
{\small 
\begin{equation}
\label{eq:uni_At_m_Atm1}
\begin{aligned}
& (
\tfrac{1}{(A_k^{t,s})^{\tfrac{1+\nu}{1+3\nu}}}
-
\tfrac{1}{(A_k^{t-1,s})^{\tfrac{1+\nu}{1+3\nu}}}
 )
 (
\tfrac{1}{(A_k^{t,s})^{\tfrac{2\nu}{1+3\nu}}}
+
\tfrac{1}{(A_k^{t-1,s})^{\tfrac{2\nu}{1+3\nu}}}
 )
>
\tfrac{1}{A_k^{t,s}}-\tfrac{1}{A_k^{t-1,s}}.
\end{aligned}
\end{equation}}
Also, using the relation $A_k^{t,s}=M_k^{t,s}(\alpha_k^{t,s})^2$ in \eqref{eq:AGamma} with the bound on $M_k^{t,s}$ in Proposition~\ref{thm:universal_estimates_bound}, we obtain,
\(
\alpha_k^{t,s}(A_k^{t,s})^{-\tfrac{1+\nu}{1+3\nu}}
\ge
 [
\tfrac{1}{2}\,
(\overline{M}_{\nu,\varepsilon})^{-1}
(\gamma_k^s)^{\tfrac{1-\nu}{1+\nu}}
 ]^{\tfrac{1+\nu}{1+3\nu}},
\)
for $1\le t\le T_k^s-1$.
Then, the above relation and \eqref{eq:uni_At_m_Atm1} indicate that for $2\le t\le T_k^s-1$,
$
1/(A_k^{t,s})^{\tfrac{1+\nu}{1+3\nu}}
-
1/(A_k^{t-1,s})^{\tfrac{1+\nu}{1+3\nu}}
>
\tfrac{1}{2}\alpha_k^{t,s}(A_k^{t,s})^{-\tfrac{1+\nu}{1+3\nu}}
\ge
\tfrac{1}{2}
 [
\tfrac{1}{2}\,
(\overline{M}_{\nu,\varepsilon})^{-1}
(\gamma_k^s)^{\tfrac{1-\nu}{1+\nu}}
 ]^{\tfrac{1+\nu}{1+3\nu}},
$
where the first inequality follows also from $A_k^{t,s}< A_k^{t-1,s}$ together with the recursion in \eqref{eq:AGamma_recur}.
Moreover, at $t=1$, $A_k^{1,s}=M_k^{1,s}$, so by Proposition~\ref{thm:universal_estimates_bound},
$
1/(A_k^{1,s})^{\tfrac{1+\nu}{1+3\nu}}
=
1/(M_k^{1,s})^{\tfrac{1+\nu}{1+3\nu}}
\ge
 [
\tfrac{1}{2}\,
(\overline{M}_{\nu,\varepsilon})^{-1}
(\gamma_k^s)^{\tfrac{1-\nu}{1+\nu}}
 ]^{\tfrac{1+\nu}{1+3\nu}}.
$
Combining the above bounds, we obtain
\begin{align*}
 \tfrac{1}{(A_k^{t,s})^{\tfrac{1+\nu}{1+3\nu}}}
\ge
 \tfrac{1}{(A_k^{t-1,s})^{\tfrac{1+\nu}{1+3\nu}}}
&=
 \tfrac{1}{(A_k^{1,s})^{\tfrac{1+\nu}{1+3\nu}}}
+
\sum\nolimits_{i=2}^{t-1}
 [
 \tfrac{1}{(A_k^{i,s})^{\tfrac{1+\nu}{1+3\nu}}}
-
 \tfrac{1}{(A_k^{i-1,s})^{\tfrac{1+\nu}{1+3\nu}}}
 ]
\\
&\ge
 \tfrac{t}{2}
 [
\tfrac{1}{2}\,
(\overline{M}_{\nu,\varepsilon})^{-1}
(\gamma_k^s)^{\tfrac{1-\nu}{1+\nu}}
 ]^{\tfrac{1+\nu}{1+3\nu}},\ \forall t\ge 2.
\end{align*}
Rearranging the above two inequalities immediately yields the upper bound in \eqref{eq:universal_A_bound}.

With \eqref{eq:universal_A_bound}, we can immediately conclude that Subroutine \ref{alg:universal_subroutine_linesearch} will terminate after a finite number of iterations. Otherwise, with $t\to \infty$ the bound \eqref{eq:universal_A_bound} would yield $L_k^s\le 0$, which is impossible since Algorithm \ref{alg:universal_linesearch} always produces $L_k^s>0$. Noting that  Subroutine~\ref{alg:universal_subroutine_linesearch} terminates when $A_k^{t,s}=L_k^s$, we conclude that there exists $T_k^s$ such that $A_k^{T_k^s,s}=L_k^s$, which happens exactly during termination.
\end{proof}

\subsection{Complexity of UGS}\label{sec:UGS_complexity_analysis}
Theorem~\ref{thm:universal_parameters_and_convergence}
shows that, once Subroutine~\ref{alg:universal_subroutine_linesearch} enforces
the terminal identity
\(A_k^{T_k^s,s}=L_k^s\), the accuracy of
Algorithm~\ref{alg:universal_linesearch} is determined by
\(\Gamma_N^{S_N}\). 
The following Proposition~\ref{thm:universal_GammaA_bound} then provides bounds on
\(\Gamma_N^{S_N}\), \(\Gamma_{N-1}^{S_{N-1}}\) and $\sum_{s=1}^{S_1}T_1^s$. The first bound controls the needed number of outer iterations, whereas the latter two will be needed later when counting the
number of evaluations of \(f'\). The proof of Proposition~\ref{thm:universal_GammaA_bound} is deferred to the appendix.
\begin{proposition}\label{thm:universal_GammaA_bound}
After running $N$ outer iterations in Algorithm~\ref{alg:universal_linesearch}, we have the following bounds: $\Gamma_N^{S_N}\le 8L/(N+1)^2$, $\forall N\ge1$
and
\begin{align}
    \label{eq:Gamma_bound}
    \Gamma_{N-1}^{S_{N-1}}
\le
2\,
\overline{M}_{\nu,\varepsilon}
 (8+4\sqrt{2} )^{\tfrac{1+3\nu}{1+\nu}}
 (\tfrac{1+\sqrt{5}}{2} )^{\tfrac{1-\nu}{1+\nu}}
/(\sum\nolimits_{k=2}^N\sum\nolimits_{s=1}^{S_k}T_k^s )^{\tfrac{1+3\nu}{1+\nu}}, \forall N\ge2
.
\end{align}
Moreover, the number of inner iterations in the first outer iteration is bounded by
\begin{align}\label{eq:universal_sumT_bound_k1}
\sum\nolimits_{s=1}^{S_1}T_1^s
\le
(2\overline{M}_{\nu,\varepsilon})^{\tfrac{1+\nu}{1+3\nu}}
 (4+2\sqrt{2})
/(L_0)^{\tfrac{1+\nu}{1+3\nu}}.
\end{align}
\end{proposition}

We next estimate the numbers of evaluations of \(\nabla g\) and \(f'\) required by Algorithm~\ref{alg:universal_linesearch} to complete \(N\) outer iterations. The corresponding bounds are established in Lemmas~\ref{thm:additional_geval_dueto_L} and~\ref{thm:additional_feval_dueto_M}, respectively; the proof of Lemma \ref{thm:additional_feval_dueto_M} is deferred to appendix.
\begin{lemma}\label{thm:additional_geval_dueto_L}
    To complete $N$  outer iterations, Algorithm~\ref{alg:universal_linesearch} requires at most 
    $N+\log_2(2L/L_0)$ evaluations of $\nabla g$.
\end{lemma}
\begin{proof}
By the description in Algorithm~\ref{alg:universal_linesearch}, the number of evaluations of $\nabla g$ is $\sum\nolimits_{k=1}^N S_k 
=  
\sum\nolimits_{k=1}^N (1+\log_2 L_k/L_{k-1} )
\le  N + \log_2 (2L/L_0)$.
\end{proof}
\begin{lemma}\label{thm:additional_feval_dueto_M}
    To complete $N$  outer iterations, Algorithm~\ref{alg:universal_linesearch} with Subroutine~\ref{alg:universal_subroutine_linesearch} requires at most 
        \(
\sum\nolimits_{k=1}^N\sum\nolimits_{s=1}^{S_k}T_k^s
      +
     \tfrac{2(1+\nu)}{3\nu+1} 
     \sum\nolimits_{k=1}^N S_k\log_2 (\tfrac{2\overline{M}_{\nu,\varepsilon}k^{\tfrac{1-\nu}{1+\nu}}}{L_0} )
    \)
    evaluations of $f'$.
\end{lemma} 

With the help of Proposition~\ref{thm:universal_GammaA_bound} and Lemmas~\ref{thm:additional_geval_dueto_L} and~\ref{thm:additional_feval_dueto_M}, we now establish the numbers of evaluations of \(\nabla g\) and \(f'\) required by the UGS algorithm to obtain an \(\varepsilon\)-solution in Theorem~\ref{thm:universal_overall_complexity_for_naive}. We then discuss the resulting gradient complexity.
\begin{theorem}\label{thm:universal_overall_complexity_for_naive}
To compute an
\(\varepsilon\)-solution, 
Algorithm~\ref{alg:universal_linesearch} with
Subroutine~\ref{alg:universal_subroutine_linesearch} requires at most
\(
(4\sqrt{2}+2)\sqrt{LV^*/\varepsilon}
\)
evaluations of $\nabla g$, 
and 
\begin{equation*}
C(\tfrac{\overline{M}_{\nu,\varepsilon}V^*}{\varepsilon})^{\tfrac{1+\nu}{1+3\nu}}
+
(4\sqrt{2}+2)(\tfrac{LV^*}{\varepsilon})^{\tfrac{1}{2}}
[\tfrac{7-3\nu}{3\nu+1} + 
\tfrac{2+2\nu}{3\nu+1}
\log_2 (\tfrac{\overline{M}_{\nu,\varepsilon}V^*}{\varepsilon})+\tfrac{1-\nu}{3\nu+1}\log_2(\tfrac{LV^*}{\varepsilon})]
\end{equation*}
evaluations of $f'$, where $C = (4+2\sqrt{2})
[1+4(1+\sqrt{5})^{\tfrac{1-\nu}{1+3\nu}}]2^{\tfrac{1+\nu}{1+3\nu}}$ and $V^*:=V(x_0,x^*)$. 
\end{theorem}
\begin{proof}
It follows from Lemma~\ref{thm:recursion_property}\ref{item:lambda_in_01_Lambda_decrease}
and Proposition~\ref{thm:universal_GammaA_bound} that
\(\Gamma_k^{S_k}\) decreases strictly to \(0\) as \(k\to\infty\).
Let $N_\varepsilon\ge1$ be the first outer iteration index such that $\Gamma_{N_\varepsilon}^{S_{N_\varepsilon}}\le \varepsilon/(4V^*)$.
Theorem~\ref{thm:universal_parameters_and_convergence} with  
Proposition~\ref{thm:Ak1s_ge_Lks} indicates 
that $\xub_{N_\varepsilon}$ is an $\varepsilon$-solution and by 
$\Gamma_{N_\varepsilon}^{S_{N_\varepsilon}}\le 8L/(N_\varepsilon+1)^2$ in 
Proposition~\ref{thm:universal_GammaA_bound} we know $N_\varepsilon\le \lceil \sqrt{32LV^*/\varepsilon}-1\rceil$. Hence by $L_0\ge \varepsilon/V^*$ and Lemma~\ref{thm:additional_geval_dueto_L}, the required number of evaluations of $\nabla g$ is at most $\sum_{k=1}^{N_\varepsilon}S_k\le \sqrt{32LV^*/\varepsilon}+\log_2 (2LV^*/\varepsilon)
\le 
(4\sqrt{2}+2)\sqrt{LV^*/\varepsilon}$.
Applying  $L_0\ge \varepsilon/V^*$ to $\eqref{eq:universal_sumT_bound_k1}$
yields $\sum_{s=1}^{S_1}T_1^s
\le 
(4+2\sqrt{2}) (2\overline{M}_{\nu,\varepsilon}V^*/\varepsilon)^{\tfrac{1+\nu}{1+3\nu}}$. 
If $N_\varepsilon\ge2$, then  
$\Gamma_{N_\varepsilon-1}^{S_{N_\varepsilon-1}}>\varepsilon/(4V^*)$ and hence 
\(
\sum\nolimits_{k=2}^{N_\varepsilon}\sum\nolimits_{s=1}^{S_k}T_k^s 
\le
(2\overline{M}_{\nu,\varepsilon}V^*/\varepsilon)^{\tfrac{1+\nu}{1+3\nu}}(16+8\sqrt{2})(1+\sqrt{5})^{\tfrac{1-\nu}{1+3\nu}}
\) 
by \eqref{eq:Gamma_bound}. 
Moreover, by the above upper bound on $N_\varepsilon$ and $L_0\ge \varepsilon/V^*$, we have that 
\(
   \log_2 ((\tfrac{2\overline{M}_{\nu,\varepsilon}}{L_0})k^{\tfrac{1-\nu}{1+\nu}})
\le
\log_2
((\tfrac{2\overline{M}_{\nu,\varepsilon}V^*}{\varepsilon})N_\varepsilon^{\tfrac{1-\nu}{1+\nu}})
\le
\tfrac{7-3\nu}{2+2\nu} + 
\log_2 (\tfrac{\overline{M}_{\nu,\varepsilon}V^*}{\varepsilon})+\tfrac{1-\nu}{2+2\nu}\log_2(\tfrac{LV^*}{\varepsilon})
\), for $k\le N_\varepsilon$.
Applying the above relations to Lemma~\ref{thm:additional_feval_dueto_M} completes the proof.
\end{proof}

We conclude this subsection with a few comments on Theorem~\ref{thm:universal_overall_complexity_for_naive}. 
First, the total number of
evaluations of \(\nabla g\) is bounded by
\(
    \mathcal O(
    \sqrt{LV^*/\varepsilon})
\), matching the optimal complexity bound. Second, when $\nu\in [0,1)$, using the assumption that $L\le \overline{M}_{\nu,\varepsilon}$ and the fact that $Q^{1/2}\log_2 Q \le  \mathcal{O}(Q^{(1+\nu)/(1+3\nu)})$ for all positive $Q$, 
the total number of evaluations of \(f'\) is bounded by
\(
    \mathcal O(
    (
    \overline M_{\nu,\varepsilon}V^*/\varepsilon
    )^{{(1+\nu)}/{(1+3\nu)}}
    )
    =
    \mathcal O(
    M_\nu (V^*)^{{(1+\nu)}/{2}}
    /\varepsilon
    )^{{2}/{(1+3\nu)}}
)
\) .
Such a bound matches the optimal order for the \((M_\nu,\nu)\)-Hölder component.
Third, when \(\nu=1\), the total number of evaluations of $f'$ is bounded by $\mathcal O(
    \sqrt{M_{\nu=1}V^*/\varepsilon} + \sqrt{LV^*/\varepsilon}\log_2{(M_{\nu=1}V^*/\varepsilon)})$. 
There is an unfavorable extra logarithm term in the complexity, which is due to the lack of monotonicity in the sequence $\{M_{k}^{t,s}\}$ of estimated constants.
Under a slightly stronger assumption that 
\(
    L\le
    M_{\nu=1}/
    \log_2^2(M_{\nu=1}V^*/\varepsilon)
\) instead of $L\le M_{\nu=1}$ as in \eqref{eq:LMnv_assumption}, 
such $f'$ evaluation complexity becomes $\mathcal O(
    \sqrt{M_{\nu=1}V^*/\varepsilon})$ and matches the optimal complexity. Alternatively, if we already know that $\nu=1$ in advance, by modifying the algorithm and searching for $M_{\nu}$ directly, it is easy to show that the sequence of estimates of constant $M_{\nu}$ can be monotone increasing and the unfavorable logarithm can be removed in the analysis.
Finally, 
the remaining restriction in this subsection
is the initial condition
\(
    L_0\ge \varepsilon/V^*
\). 
We will remove this requirement in the next section.
\section{Parameter-free universal gradient sliding (PFUGS)}\label{sec:PFUGS}
To remove the remaining dependence of Algorithm~\ref{alg:universal_linesearch} on the unknown quantity $V^*$,
we introduce a modification that preserves the same complexity while
requiring no prior knowledge of \(\nu\), \(M_\nu\), \(L\), or \(V^*\).
We call the resulting method the parameter-free universal gradient sliding (PFUGS) algorithm, whose
main procedure is given in Algorithm~\ref{alg:universal_linesearch_decreas_increase}.

\begin{algorithm}[h]
\caption{Proposed PFUGS algorithm}
\label{alg:universal_linesearch_decreas_increase}
\begin{algorithmic}[1]
\State In Algorithm~\ref{alg:universal_linesearch}, modify the initialization and $k=1$ iteration as follows:
\State Initialization: $x_0=\xtil_0=\xub_0\in X$ and $M_0\in(0,L]$
\State Set $k=1$, $s=1$, $t=1$, $x_1^{0,1} = \xtil_1^{0,1}=x_0$, and $M_1^{0,1} = M_0$
\State Compute
\vspace{-10pt}
\begin{align}
\begin{aligned}
    &\xlb_k^s = (1-\gamma_k^s)\xub_{k-1}+\gamma_k^sx_{k-1}\nonumber
    \\&
    \xlb_k^{t,s} = (1-\gamma_k^s)\xub_{k-1} + \gamma_k^s[(1-\alpha_k^{t,s})\xtil_k^{t-1,s} +\alpha_k^{t,s}x_k^{t-1,s}]\nonumber
    \\&
    x_k^{t,s} = \operatorname{argmin}_{x\in X} \langle \nabla g(\xlb_k^s) + f'(\xlb_k^{t,s}), x\rangle + \eta_k^sV(x_{k-1},x)+p_k^{t,s}V(x_k^{t-1,s},x)
    \\&
    \xtil_k^{t,s} = (1-\alpha_k^{t,s})\xtil_k^{t-1,s} +\alpha_k^{t,s}x_k^{t,s}\nonumber
    \\&
    \xub_k^{t,s} = (1-\gamma_k^s)\xub_{k-1} + \gamma_k^s\xtil_k^{t,s}\nonumber
\end{aligned}
\end{align}
Here the parameters are set as
    $\alpha_1^{1,1}=1$,
    $\gamma_1^1=1$,
    $\eta_1^1 = L_1^1\gamma_1^1$,
    $p_1^{1,1} = L_1^1\gamma_1^1(1-\alpha_1^{1,1})/\alpha_1^{1,1} + \gamma_1^1M_1^{1,1}\alpha_1^{1,1}$,
    $L_1^1=M_1^{1,1}= 2^{i_1^{1,1}}M_1^{0,1}$, where
$i_1^{1,1}$ is the smallest nonnegative integer, such that both \eqref{eq:universal_f_ineq_s} and \eqref{eq:universal_g_ineq_s} hold
\State Set
$x_k^s = x_k^{1,s}$,
$\xtil_k^s = \xtil_k^{1,s}$,
$\xub_k^s = \xub_k^{1,s}$,
and update $s\gets s+1$
\State Set $L_k^s = L_k^1/2^{s-1}$ and $M_k^{0,s} =\max\{M_0,L_k^s\}$\label{state:universal_outeralg_k1sge2}
\State Call subroutine
$(x_k^s, \tilde x_k^s, \xub_k^s)
=
\mathcal{A}(\nabla g(\xlb_k^s), \eta_k^s, M_k^{0,s})$
where,
$\xlb_k^s = (1-\gamma_k^s)\xub_{k-1}+\gamma_k^sx_{k-1}$,
and parameters are set as in Algorithm~\ref{alg:universal_linesearch}
\State If condition \eqref{eq:universal_g_ineq_s} holds,  set $s\gets s+1$ and return to line \eqref{state:universal_outeralg_k1sge2};
otherwise, set
$S_k = s-1$,
$L_k=L_k^{S_k}$,
$x_k = x_k^{S_k}, \xtil_k = \xtil_k^{S_k}, \xub_k=\xub_k^{S_k}$, and $k=2$\label{state:pfugs_k1_ineq}
\end{algorithmic}
\end{algorithm}
 Algorithm~\ref{alg:universal_linesearch_decreas_increase} differs from Algorithm~\ref{alg:universal_linesearch} only in the
linesearch order in the first iteration. So all
convergence results established in Subsection~\ref{sec:universal_base} and \ref{sec:universal_convergence_properties}
remain valid. 
We also note that the initial condition $M_0\le L$ requires no knowledge of $L$; see the remark on satisfying $L_0\le L$ after the description of Algorithm \ref{alg:GDS_linesearch_naive}.
We now derive bounds of $\Gamma_N^{S_N}$ and $\Gamma_{N-1}^{S_{N-1}}$ in Proposition~\ref{thm:universal_GammaA_bound_final_version}; its proof is deferred to appendix.
\begin{proposition}
\label{thm:universal_GammaA_bound_final_version}
After running $N$ outer iterations in Algorithm~\ref{alg:universal_linesearch_decreas_increase}, we have the following bounds:  
$\Gamma_N^{S_N}\le 8L/(N+1)^2$, $\forall N\ge1$,
and if $N\ge2$, \\
\(
\Gamma_{N-1}^{S_{N-1}}
\le
(16+8\sqrt{2})^{\tfrac{1+3\nu}{1+\nu}}
2^{\tfrac{2\nu}{1+\nu}}
(1+\sqrt5)^{\tfrac{1-\nu}{1+\nu}}
\overline{M}_{\nu,\varepsilon}
/
\left(
\sum\nolimits_{s=1}^{S_1+1}T_1^s
+
\sum\nolimits_{k=2}^N\sum\nolimits_{s=1}^{S_k}T_k^s
\right)^{\tfrac{1+3\nu}{1+\nu}}
.
\)
\end{proposition}

We next estimate the number of evaluations of $\nabla g$ and $f'$ required by Algorithm~\ref{alg:universal_linesearch_decreas_increase} to complete $N$ outer iterations. The corresponding bounds are established in Lemma~\ref{thm:additional_geval_dueto_L_final_version} and \ref{thm:additional_feval_dueto_M_alg_final_version}, respectively; the proof of Lemma \ref{thm:additional_feval_dueto_M_alg_final_version} is deferred to appendix.
\begin{lemma}\label{thm:additional_geval_dueto_L_final_version}
Algorithm~\ref{alg:universal_linesearch_decreas_increase} requires only one evaluation of $\nabla g$ to complete the first iteration. For $N\ge2$ it requires at most  
$N+\log_2 (2L/L_1 )$ evaluations of $\nabla g$.
\end{lemma}
\begin{proof}
    Notice that $\gamma_1^s=1$, so the $\nabla g$ evaluation points at $k=1$ satisfy $\xlb_1^s=x_0$, $\forall s\ge1$, that is the same value $\nabla g(x_0)$ is reused for all trials $L_1^s$. Hence, Algorithm~\ref{alg:universal_linesearch_decreas_increase} needs only one evaluation of $\nabla g$ at the first iteration. To complete $N\ge2$ iterations, by $L_k\le2L$, it requires at most 
    $
     1
    +
     \sum\nolimits_{k=2}^{N} (1+
    \log_2(L_k/L_{k-1}))
    \le
    N+\log_2 (2L/L_1 ).
    $
\end{proof}
\begin{lemma}\label{thm:additional_feval_dueto_M_alg_final_version}
    For $N\ge2$,
    the total number of evaluations of $f'$ required by 
Algorithm~\ref{alg:universal_linesearch_decreas_increase} is at most
$
 ( \sum\nolimits_{s=1}^{S_1+1}T_1^{s}       + \sum\nolimits_{k=2}^N \sum\nolimits_{s=1}^{S_k}T_k^s )
+
 \sum\nolimits_{k=1}^N S_k 
    \tfrac{2(1+\nu)}{3\nu+1}
    \log_2 ((\tfrac{4 \overline{M}_{\nu,\varepsilon}}{L_1})k^{\tfrac{1-\nu}{1+\nu}} ).
$
\end{lemma}

With the help of Proposition \ref{thm:universal_GammaA_bound_final_version} and Lemmas \ref{thm:additional_geval_dueto_L_final_version} and \ref{thm:additional_feval_dueto_M_alg_final_version}, we present the needed number of evaluations of $\nabla g$ and $f'$ to get an $\varepsilon$-solution by PFUGS in Theorem \ref{thm:universal_overall_complexity_final_version}.

\begin{theorem}\label{thm:universal_overall_complexity_final_version}
To compute an $\varepsilon$-solution,
Algorithm~\ref{alg:universal_linesearch_decreas_increase} with Subroutine~\ref{alg:universal_subroutine_linesearch} requires at most $
(4\sqrt{2}+3)\sqrt{LV^*/\varepsilon}$ 
evaluations of $\nabla g$ and 
\(
D(\tfrac{\overline{M}_{\nu,\varepsilon}V^*}{\varepsilon})^{\tfrac{1+\nu}{1+3\nu}}
    +
    [(4\sqrt{2}+3)(\tfrac{LV^*}{\varepsilon})^{\tfrac{1}{2}}+\log_2(\tfrac{8\overline{M}_{\nu,\varepsilon}V^*}{\varepsilon})]
     [\tfrac{13+3\nu}{3\nu+1} 
     +  
\tfrac{2+2\nu}{3\nu+1} \log_2(\tfrac{\overline{M}_{\nu,\varepsilon}V^*}{\varepsilon}) 
+ 
 \tfrac{1-\nu}{3\nu+1}\log_2(\tfrac{LV^*}{\varepsilon})]
\)
evaluations of $f'$, where 
\(
D:=(32+16\sqrt{2})(1+\sqrt{5})^{\tfrac{1-\nu}{1+3\nu}}2^{\tfrac{1+\nu}{1+3\nu}}
\) and $V^*:=V(x_0,x^*)$.
\end{theorem}
\begin{proof}
    At the first outer iteration, Algorithm~\ref{alg:universal_linesearch_decreas_increase} performs a decreasing search for an estimate of $L$ over $s\ge2$. During this decreasing search process, if, at line~\ref{state:pfugs_k1_ineq} in Algorithm~\ref{alg:universal_linesearch_decreas_increase}, the smoothness inequality \eqref{eq:universal_g_ineq_s} is tested and not violated at some trial value $L_1^{s_\varepsilon}$ such that $\varepsilon/(8V^*)<L_1^{s_\varepsilon}\le \varepsilon/(4V^*)$, then Theorem~\ref{thm:universal_parameters_and_convergence} with Proposition~\ref{thm:Ak1s_ge_Lks} indicates that $\xub_1:=\xub_1^{s_\varepsilon}$ is an $\varepsilon$-solution, and only one  evaluation of $\nabla g$ is needed by Lemma~\ref{thm:additional_geval_dueto_L_final_version}. 
    Moreover, we make the following observations. First, by similar arguments in \eqref{eq:pfugs_sumT1s} we have \(
\sum\nolimits_{s=1}^{s_\varepsilon} T_1^s \le 
2(2\overline{M}_{\nu,\varepsilon})^{\tfrac{1+\nu}{1+3\nu}}
\sum\nolimits_{s=1}^{s_\varepsilon}\gamma_1^s/(\Gamma_1^s)^{\tfrac{1+\nu}{1+3\nu}}
\le
2(2\overline{M}_{\nu,\varepsilon}/\Gamma_1^{s_\varepsilon})^{\tfrac{1+\nu}{1+3\nu}}
/(1-(2)^{-\tfrac{1+\nu}{1+3\nu}})
\). Second, the number of tested estimates of $L$ in the decreasing process is  
$s_\varepsilon-1 = \log_2(M_1^{1,1}/L_1^{s_\varepsilon})\le 
    \log_2(2\overline{M}_{\nu,\varepsilon}/\Gamma_1^{s_\varepsilon})
    $ by Proposition~\ref{thm:universal_estimates_bound}, $\gamma_1^s=1$ and $\Gamma_1^s=L_1^s$. Third, by similar arguments to \eqref{eq:universal_log_over_log}, we have $\log_2\tfrac{M_1^{T_1^s-1,s}}{M_1^{0,s}}
    \le
    \tfrac{2(1+\nu)}{1+3\nu}\log_2\tfrac{2\overline{M}_{\nu,\varepsilon}}{\Gamma_1^{s_\varepsilon}}
    $. Hence, by similar arguments of \eqref{eq:pfugs_total_feval} and the above observations, the needed number of evaluations of $f'$ is at most 
    \(
\sum_{s=1}^{s_\varepsilon}T_1^s
+\sum_{s=2}^{s_\varepsilon}
\log_2\tfrac{M_1^{T_1^s-1,s}}{M_1^{0,s}}
\le   
2(\tfrac{16\overline{M}_{\nu,\varepsilon}V^*}{\varepsilon})^{\tfrac{1+\nu}{1+3\nu}}/(1-(2)^{-\tfrac{1+\nu}{1+3\nu}})
+
\tfrac{2(1+\nu)}{1+3\nu}
\log_2^2\tfrac{16\overline{M}_{\nu,\varepsilon}V^*}{\varepsilon}
\), which is covered by the bound in the statement.
Otherwise, the decreasing search process accepts some $L_1>\varepsilon/(4V^*)$. Let $N_\varepsilon\ge2$ be the outer index such that $\Gamma_{N_\varepsilon}^{S_{N_\varepsilon}}
\le 
\varepsilon/(4V^*)
<
\Gamma_{N_\varepsilon-1}^{S_{N_\varepsilon-1}}
$. Theorem~\ref{thm:universal_parameters_and_convergence} with Proposition~\ref{thm:Ak1s_ge_Lks} indicates that $\xub_{N_\varepsilon}$ is an $\varepsilon$-solution and by $\Gamma_{N_\varepsilon}^{S_{N_\varepsilon}}\le 8L/(N_\varepsilon+1)^2$ in Proposition~\ref{thm:universal_GammaA_bound_final_version} we have $N_\varepsilon\le\lceil\sqrt{32LV^*/\varepsilon}-1\rceil
$. Hence  Lemma~\ref{thm:additional_geval_dueto_L_final_version} indicates that the number of evaluations of $\nabla g$ is at most $\sqrt{32LV^*/\varepsilon} + \log_2 (8LV^*/\varepsilon)\le
(4\sqrt{2}+3)\sqrt{LV^*/\varepsilon}$. 
Moreover, from Proposition~\ref{thm:universal_GammaA_bound_final_version} we have 
 \(
 \sum\nolimits_{s=1}^{S_1+1}T_1^s + \sum\nolimits_{k=2}^{N_\varepsilon}\sum\nolimits_{s=1}^{S_k}T_k^s 
 \le 
 (32+16\sqrt{2})(1+\sqrt{5})^{\tfrac{1-\nu}{1+3\nu}}(\tfrac{2\overline{M}_{\nu,\varepsilon}V^*}{\varepsilon})^{\tfrac{1+\nu}{1+3\nu}}.
 \)
 By the above bound on $N_\varepsilon$, we have that 
 \(
 \log_2((\tfrac{4\overline{M}_{\nu,\varepsilon}}{L_1})k^{\tfrac{1-\nu}{1+\nu}}) 
 \le
 \log_2((\tfrac{16\overline{M}_{\nu,\varepsilon}V^*}{\varepsilon})N_\varepsilon^{\tfrac{1-\nu}{1+\nu}}) 
 \le 
 \tfrac{13+3\nu}{2+2\nu} + 
\log_2(\tfrac{\overline{M}_{\nu,\varepsilon}V^*}{\varepsilon}) 
+ 
 \tfrac{1-\nu}{2+2\nu}\log_2(\tfrac{LV^*}{\varepsilon})
 \), for $k\le N_\varepsilon$. From analysis in Lemma~\ref{thm:additional_geval_dueto_L_final_version} we have $\sum_{k=1}^{N_\varepsilon}S_k\le N_\varepsilon+\log_2(2L/L_1) + \log_2(M_1^{1,1}/L_1)
 \le 
(4\sqrt{2}+3)\sqrt{LV^*/\varepsilon}
+
\log_2(8\overline{M}_{\nu,\varepsilon}V^*/\varepsilon)
$.
Applying the above relations to Lemma~\ref{thm:additional_feval_dueto_M_alg_final_version} completes this proof.
\end{proof}
We conclude this section by noting that the discussion regarding  complexities after  Theorem~\ref{thm:universal_overall_complexity_for_naive} applies to 
Theorem~\ref{thm:universal_overall_complexity_final_version} as well.
\section{Numerical Experiments}
\label{sec:numerical_experiments}
In this section, we conduct two numerical experiments designed to assess the practical performance of PFUGS, focusing in particular on its sliding mechanism in Subsection~\ref{sec:experiment1} and parameter-free nature in Subsection~\ref{sec:experiment2}. All experiments were performed in MATLAB R2025b on an iMac with Apple M1 chip and 16 GB RAM.
Code is available at \url{https://github.com/YanWu0/PFUGS}.
\subsection{Penalized matrix game}
\label{sec:experiment1}
We first consider the following penalized matrix-game problem. Given a payoff matrix $A\in\mathbb{R}^{n\times m}$, a linear operator $B\in\mathbb{R}^{p\times n}$, and a vector $d\in\mathbb{R}^p$, we seek a mixed strategy $x$ that balances the matrix-game objective with a quadratic penalty term. The problem is formulated as $\Phi^* =\min\nolimits_{x \in \Delta_n} \Phi(x)
:= g(x)+f(x)=
\lambda \|Bx-d\|_2^2 + \max\nolimits_{y \in \Delta_m}\langle x, Ay\rangle$,
where $\Delta_{(\cdot)}$ denotes the standard simplex. 

In this experiment, we set $n=2000$, $m=500$, $p=2000$, and $\lambda=0.05$. Linear operators $A$, $B$, and $d$ are generated randomly with entries drawn independently from uniform $[-1,1]$. 
The penalty term $g$ is more expensive to evaluate due to its higher dimensionality, while
$\lambda$ controls its curvature. Such configuration makes it advantageous to process $g$ in the main PFUGS algorithm while handling $f$ in subroutine.

We use Nesterov's fast gradient method (FGM) \cite{nesterov2015universal} as the benchmark. Since the feasible set is simplex, both PFUGS and FGM use the entropy distance-generating function, whose associated prox-function is the Kullback--Leibler divergence. To compare their performance, we generate ten independent instances and initialize both algorithms at the uniform point $x_0=(1/n,\ldots,1/n)^\top$. Following the experimental setup used for FGM in \cite{nesterov2015universal}, we test accuracies $\varepsilon\in\{2^{-5},\dots,2^{-10}\}$. For each $\varepsilon$, both algorithms run until they first produce an iterate $\xub_k$ satisfying $\Phi(\xub_k)-\Phi^*\le \varepsilon$, where the reference value $\Phi^*$ is computed by \texttt{quadprog}. Table~\ref{table:experiment1} reports the mean of ten instances, with standard deviations shown in parentheses, for running time, achieved objective gap, number of (sub)gradient evaluations of $f$ (``fsubg'') and $g$ (``ggrad'').
\begin{table}[htbp]
\centering
\caption{Average performance of PFUGS and FGM}
\label{table:experiment1}
\scriptsize
\setlength{\tabcolsep}{4pt}
\renewcommand{\arraystretch}{1.15}
\begin{tabular}{c c c c c c}
\hline
$\varepsilon$ & Algorithm & Running time (s) & Gap & fsubg & ggrad \\
\hline
\multirow{2}{*}{$2^{-5}$}
& PFUGS & $0.4(\pm0.03)$ & $3.0{\times}10^{-2}(\pm6.5{\times}10^{-4})$ & $1377(\pm106)$ & $24(\pm1)$ \\
& FGM   & $2.8(\pm0.4)$ & $3.1{\times}10^{-2}(\pm5.2{\times}10^{-5})$ & $607(\pm93)$ & $607(\pm93)$ \\
\hline
\multirow{2}{*}{$2^{-6}$}
& PFUGS & $1.2(\pm0.1)$ & $1.5{\times}10^{-2}(\pm1.9{\times}10^{-4})$ & $5214(\pm401)$ & $35(\pm5)$ \\
& FGM   & $11.4(\pm1.6)$ & $1.6{\times}10^{-2}(\pm4.0{\times}10^{-6})$ & $2431(\pm340)$ & $2431(\pm340)$ \\
\hline
\multirow{2}{*}{$2^{-7}$}
& PFUGS & $4.6(\pm 0.4)$ & $7.7{\times}10^{-3}(\pm9.5{\times}10^{-5})$ & $21543(\pm1583)$ & $49(\pm1)$ \\
& FGM   & $46.1(\pm6.1)$ & $7.8{\times}10^{-3}(\pm9.3{\times}10^{-7})$ & $9830(\pm1342)$ & $9830(\pm1342)$ \\
\hline
\multirow{2}{*}{$2^{-8}$}
& PFUGS & $18.7(\pm1.6)$ & $3.9{\times}10^{-3}(\pm3.4{\times}10^{-5})$ & $90822(\pm7627)$ & $70(\pm1)$ \\
& FGM   & $167.6(\pm24.1)$ & $3.9{\times}10^{-3}(\pm1.2{\times}10^{-7})$ & $39906(\pm5499)$ & $39906(\pm5499)$ \\
\hline
\multirow{2}{*}{$2^{-9}$}
& PFUGS & $77.0(\pm6.8)$ & $1.9{\times}10^{-3}(\pm1.1{\times}10^{-5})$ & $382500(\pm31704)$ & $102(\pm16)$ \\
& FGM   & $680.8(\pm99.9)$ & $2.0{\times}10^{-3}(\pm1.1{\times}10^{-8})$ & $161332(\pm21958)$ & $161332(\pm21958)$ \\
\hline
\multirow{2}{*}{$2^{-10}$}
& PFUGS & $318.6(\pm29.9)$ & $9.7{\times}10^{-4}(\pm3.4{\times}10^{-6})$ & $1587005(\pm139703)$ & $145(\pm23)$ \\
& FGM   & $2698.9(\pm363.0)$ & $9.8{\times}10^{-4}(\pm1.2{\times}10^{-9})$ & $649097(\pm88131)$ & $649097(\pm88131)$ \\
\hline
\end{tabular}
\begin{flushleft}
\footnotesize
Note: Values are reported as mean ($\pm$ standard deviation) over ten independent instances.
\end{flushleft}
\end{table}

Table~\ref{table:experiment1} shows that PFUGS consistently outperforms FGM in wall-clock time across all tested accuracies. The improvement becomes more pronounced as the target accuracy tightens. This speedup is mainly due to the dramatic reduction in gradient evaluations of the penalty term $g$. Although PFUGS performs more subgradient evaluations of the max term $f$, this tradeoff is favorable in the present regime. Overall, the results clearly demonstrate the practical advantage of the sliding mechanism.
\subsection{Smoothed penalized matrix game}
\label{sec:experiment2}
We next consider a smoothed variant of the penalized matrix-game problem. Starting from the model in Experiment~\ref{sec:experiment1}, we replace the max term by its entropy-smoothed counterpart
\(
\tilde f(x)
:=
\mu \log \sum_{j=1}^m \exp\!\bigl(\langle A_j,x\rangle/\mu\bigr),
\)
where $\mu>0$ is the smoothing parameter and $A_j$ denotes the $j$-th column of $A$. The resulting problem is
\(
\Phi_\mu^\star
:=
\min_{x\in\Delta_n}\Phi_\mu(x)
=
g(x)+\tilde f(x).
\)

We use the same data configuration, uniform initialization, and entropy prox setup as in Experiment~\ref{sec:experiment1}, with smoothing parameter  $\mu=0.0001$, target tolerance $\varepsilon=2^{-10}$ and reference value $\Phi^*_\mu$ computed by \texttt{fmincon}.
For comparison, we use accelerated gradient sliding method (AGS) \cite{lan2022accelerated} as a benchmark. Unlike PFUGS, AGS requires prescribed smoothness constants for the two objective components. 
In the entropy setting, the relevant smoothness constant for penalty term $g$ is 
\(
L_g
=
2\lambda \max_{i,j}|(B^\top B)_{ij}|
\), which can be computed exactly. 
For component $\tilde f$, we use the computable reference scale
\(
M_{\tilde{f}}
=
\max_{1\le i\le n}(\max_j A_{ij}-\min_j A_{ij})^2/4\mu,
\) which is an upper bound on the true smoothness constant of $\tilde{f}$.
To assess the sensitivity of AGS to parameter selection, we test it over all nine pairs 
\(
(\hat{L}_g, \hat{M}_{\tilde{f}})
\in 
\{L_g/100,\;L_g,\;100L_g\}
\times
\{M_{\tilde{f}}/100,\;M_{\tilde{f}},\;100M_{\tilde{f}}\}.
\)
For each pair, AGS is compared with the same PFUGS trajectory by plotting the gap 
$\Phi_\mu(\xub_k)-\Phi^*_\mu$ against running time (in seconds). The panel labels in Figure~\ref{fig:experiment2} describe the values of $(\hat{L}_g, \hat{M}_{\tilde{f}})$ and stopping status.
\begin{figure}[htbp]
\centering
\includegraphics[width=0.99\textwidth]{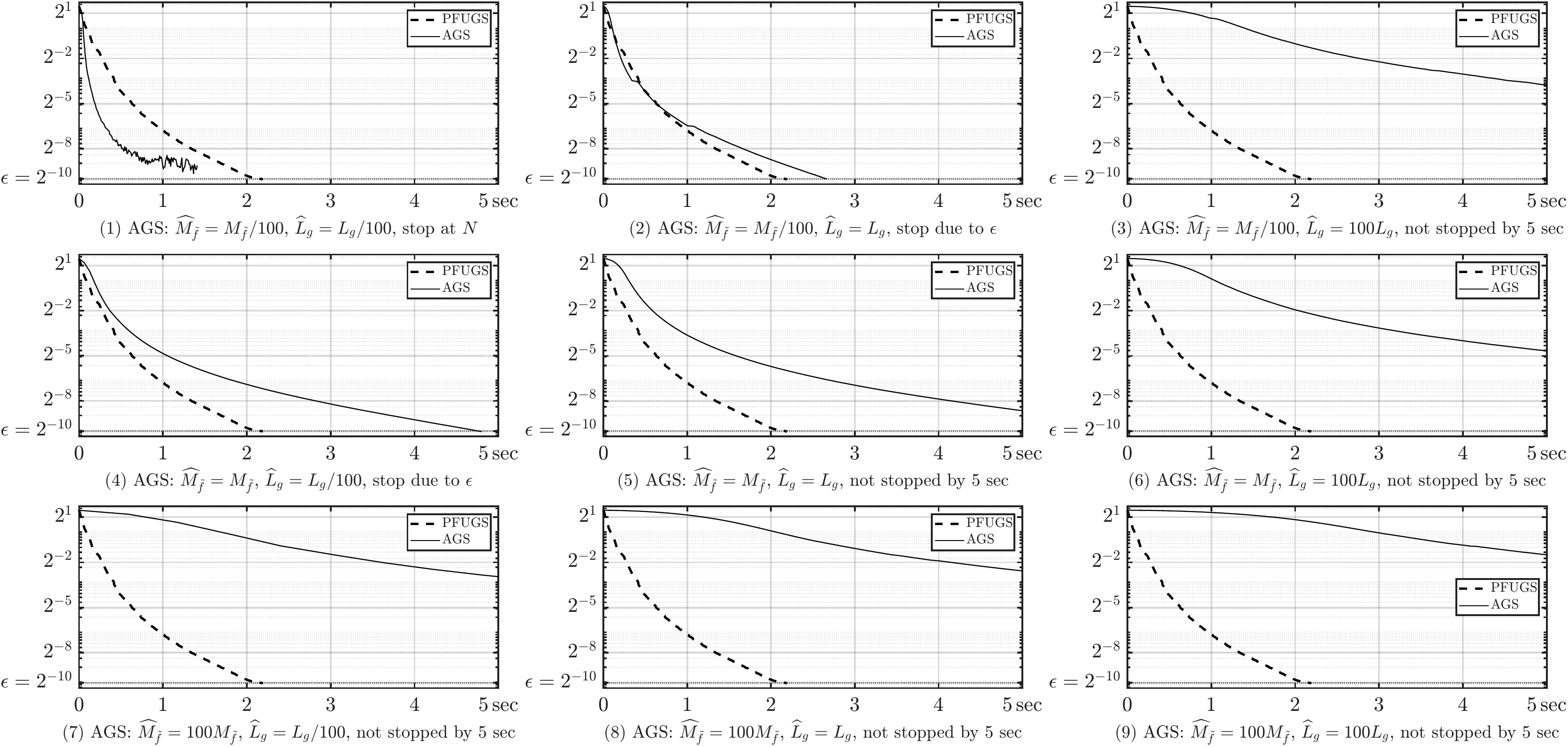}
\caption{Comparison of PFUGS and AGS under predefined smoothness constants}
\label{fig:experiment2}
\end{figure}

Figure~\ref{fig:experiment2} shows that the practical performance of AGS depends strongly on the guessed smoothness constants. Some choices lead to rapid progress (see subplot 2), while others fail to reach the target accuracy within the theoretical AGS budget associated with the guessed constants (see subplot 1) or within the displayed time horizon (see subplots 3 and 5-9). In contrast, PFUGS does not require these constants as inputs and gives a stable and fast gap reduction. These results highlight that PFUGS avoids manual parameter tuning while maintaining competitive runtime performance.
\section{Conclusion}
\label{sec:conclusion}
This work develops a PFUGS framework for convex composite optimization that adapts to the unknown smoothness parameters of  $f$ and $g$.
This algorithm computes an approximate solution using at most $\mathcal{O}\! ((M_\nu/\varepsilon)^{{2}/{(1+3\nu)}} )$ evaluations of (sub)gradients of $f$ and $\mathcal{O}\! (\sqrt{L/\varepsilon})$ evaluations of gradients of $g$.   
We believe that our results open several avenues for future research, including extension to functionally constrained problems and nonconvex composite optimization.

\appendix
\bibliographystyle{siamplain}
\bibliography{yan}
\section{Some omitted proofs}\label{sec:appendix}
\begin{proof}[Proof of Lemma \ref{thm:GDS_outer_convergence}]
Fix any $k\ge1$. The approximate optimality condition in line~\ref{state:GDS_subproblem_optimality} of Algorithm~\ref{alg:GDS} indicates
\(
    \langle \nabla g(\xtil_{k-1}), \xtil_k-x\rangle +f(\xtil_k)-f(x)
    \le 
    \epsilon_k(x) + \tfrac{\eta_k}{2}(\|x-\xtil_{k-1}\|^2 - \| x- \xtil_k\|^2 - \|\xtil_k-\xtil_{k-1}\|^2).
\)
From \eqref{eq:GDS_g_ineq} and convexity of $g(\cdot)$ we have
\(
     g(\xtil_k)-g(x)
        \le
    \langle \nabla g(\xtil_{k-1}), \xtil_k-x\rangle + \tfrac{L_k}{2}\|\xtil_k-\xtil_{k-1}\|^2.
\)
Summing the above two inequalities while recalling $\eta_k=L_k$ yields
\(
    g(\xtil_k)+ f(\xtil_k)-g(x)-f(x)\le \epsilon_k(x) + \tfrac{\eta_k}{2}(\|x-\xtil_{k-1}\|^2 - \| x- \xtil_k\|^2 ).
\)
Applying the definition of $\xub_N$ and convexity of $g(\cdot)$ and $f(\cdot)$ to the above relation concludes the lemma.
\end{proof}

\begin{proof}[Proof of Lemma \ref{thm:GDS_subroutine_convergence}]
     From \eqref{eq:GDS_f_ineq} and convexity of $f(\cdot)$ we have
  \(
        f(x_k^t)-f(x)
    \le
    \langle \nabla f(x_k^{t-1}), x_k^t-x\rangle + \tfrac{M_k^t}{2}\|x_k^t - x_k^{t-1}\|^2.
    \)
    Applying Lemma \ref{thm:three_point_Breg} to \eqref{eq:GDS_subroutine_update} yields
    \begin{align*}
        \langle \nabla f(x_k^{t-1}) +\nabla g(\xtil_{k-1}),x_k^t-x\rangle
        \le &
        \tfrac{p_k^t}{2}(\|x-x_k^{t-1}\|^2 - \|x-x_k^t\|^2 -\|x_k^t-x_k^{t-1}\|^2)
        \\
        &+
        \tfrac{\eta_k}{2}(\|x-\xtil_{k-1}\|^2 - \|x-x_k^t\|^2 -\|x_k^t-\xtil_{k-1}\|^2).
    \end{align*}
Summing the above two inequalities, multiplying both sides by $(\sum\nolimits_{\tau=1}^{T_k}1/p_k^\tau)^{-1} 1/p_k^t$,
and summing over $t = 1,\ldots,T_k$ while recalling that $p_k^t = M_k^t $ we have that 
\begin{align*}
\begin{aligned}
    &- \tfrac{\eta_k}{2} \|x - \xtil_{k-1}\|^2  + (\sum\nolimits_{\tau=1}^{T_k}\tfrac{1}{p_k^\tau})^{-1} \sum\nolimits_{t=1}^{T_k} \tfrac{1}{p_k^t} (f(x_k^t) - f(x)+ \langle \nabla g(\xtil_{k-1}), x_k^t - x \rangle)
    \\
    \le&
    (\sum\limits_{\tau=1}^{T_k}\tfrac{1}{p_k^\tau})^{-1}\sum\limits_{t=1}^{T_k} \tfrac{1}{p_k^t}
    [
        \tfrac{p_k^t}{2}
        (\|x - x_k^{t-1}\|^2 - \|x - x_k^t\|^2)
        - \tfrac{\eta_k}{2}
        (
             \|x - x_k^t\|^2
            + \|x_k^t - \xtil_{k-1}\|^2
        )
    ]
    .
    \end{aligned}
\end{align*}
Rearranging terms in the above inequality and recalling that $\xtil_k$ is a convex combination of $x_k^t$'s, we have
\begin{align*}
& f(\xtil_k)-f(x)
+ \langle\nabla g(\xtil_{k-1}), \xtil_k-x\rangle - \tfrac{\eta_k}{2}(\|x-\xtil_{k-1}\|^2-\|x-\xtil_k\|^2-\|\xtil_k-\xtil_{k-1}\|^2)
\\
\le &
(2\sum\nolimits_{\tau=1}^{T_k} 1/p_k^\tau )^{-1}(\|x-x_{k-1}\|^2-\|x-x_k\|^2) = \varepsilon_k(x).
\end{align*}
We conclude the lemma immediately by noting that the left-hand-side of the above inequality is exactly $\langle \nabla\phi_k(\xtil_k), \xtil_k - x\rangle + f(\xtil_k) - f(x)$ in the approximate optimality condition in line~\ref{state:GDS_subproblem_optimality} of Algorithm~\ref{alg:GDS}.
\end{proof}
\begin{proof}[Proof of Lemma~\ref{thm:recursion_property}]
Since $\Lambda_1=E_1(\lambda_1)^2=E_1>0$, the quadratic equation
\(
E_2(\lambda_2)^2=\Lambda_1(1-\lambda_2)
\)
admits a unique positive solution given by
\(
\lambda_2
=\tfrac{2}{1+\sqrt{1+4E_2/\Lambda_1}}.
\)
Because $\infty>E_2/\Lambda_1>0$, we have
$\lambda_2\in(0,1)$. Consequently,
\(
0<\Lambda_2=\Lambda_1(1-\lambda_2)<\Lambda_1.
\)
Assuming that $\Lambda_{\tau-1}>0$ for some $\tau\ge3$,
then by similar argument, we have $\lambda_\tau\in(0,1)$ and $0<\Lambda_\tau<\Lambda_{\tau-1}$, which completes the proof of part~\ref{item:lambda_in_01_Lambda_decrease}.

The conclusion in Part~\ref{item:AT_eq_sum_a_over_A} follows immediately from dividing both sides of relation $b_i = (1-\lambda_i)b_{i-1}+\lambda_i d_i$ by $\Lambda_i$, for $i\ge1$, and summing over $i=1,\dots,\tau$, while noticing the recursive relation $1/\Lambda_{i-1} = (1-\lambda_i)/\Lambda_i$ for all $i\ge2$ and $\lambda_1=1$ by definition.

Having established in part~\ref{item:lambda_in_01_Lambda_decrease} the existence and uniqueness of $\lambda_\tau$, part~\ref{item:lambda_Lambda_E} follows from the explicit formula
\(
\lambda_\tau
= \tfrac {2}{1+\sqrt{1+4E_\tau/\Lambda_{\tau-1}}},
 \forall \tau \ge 2,
\)
and 
\(
\Lambda_\tau = \Lambda_{\tau-1}(1-\lambda_\tau).
\)

Given nondecreasing $E_\tau$ and  strictly decreasing $\Lambda_\tau>0$ in part~\ref{item:lambda_in_01_Lambda_decrease}, we have that $E_\tau/\Lambda_{\tau-1}$ is strictly increasing for $\tau\ge2$. Then, the explicit formula of $\lambda_\tau$ indicates the decrease of $\lambda_\tau$ for $\tau\ge2$. So,
part~\ref{item:lambda_decrease} holds since $\lambda_1=1$ and $\lambda_\tau<1$ by part~\ref{item:lambda_in_01_Lambda_decrease}.

We show part~\ref{item:lambda_bound} by induction. Given $\lambda_1=1=2/(1+1)$, the conclusion in part~\ref{item:lambda_bound} holds at $\tau = 1$.
Assume that $\lambda_{\tau-1}\le 2/\tau$ for some $\tau\ge2$.
Then,
$\Lambda_{\tau-1}
= E_{\tau-1}(\lambda_{\tau-1})^2
\le 4E_\tau/\tau^2$, and hence
$\lambda_\tau
= \tfrac{2}{1+\sqrt{1+4E_\tau/\Lambda_{\tau-1}}}
\le \tfrac{2}{1+\sqrt{1+\tau^2}}\le \tfrac{2}{\tau+1}$.

For part~\ref{item:c_def}, the given information $\Lambda_{\tau-1}(1-\lambda_\tau) =\Lambda_{\tau} \le E_\text{fix}<\Lambda_{\tau-1}$ indicates that
$\lambda_\tau
\ge
1- E_\text{fix}/\Lambda_{\tau-1}>0$.
Thus,
$c
=  E_\text{fix}/[(1-E_\text{fix}/\Lambda_{\tau-1})^2 E_\tau]
\ge E_\text{fix}/[(\lambda_\tau)^2 E_\tau]
= E_\text{fix}/\Lambda_\tau
\ge 1$. Finally, the conclusion $\hat{\Lambda}_\tau = E_{\mathrm{fix}}$ follows immediately from noticing $
    \hat{\lambda}_\tau = \tfrac{-\Lambda_{\tau-1}+\sqrt{(\Lambda_{\tau-1})^2 + 4\Lambda_{\tau-1}cE_\tau}}{2cE_\tau}
$, $cE_\tau = \tfrac{\Lambda_{\tau-1}^2E_\text{fix}}{(\Lambda_{\tau-1}-E_\text{fix})^2}$ and $\hat{\Lambda}_\tau = cE_\tau(\hat{\lambda}_\tau)^2$.

The condition \(q\ge1\) indicates
\(
\sqrt{q}(1+\sqrt{1+4E_\tau/\Lambda_{\tau-1}})
\ge
1+\sqrt{1+4qE_\tau/\Lambda_{\tau-1}}.
\)
Then, the conclusion in part~\ref{item:lambda_compare_L_increase} follows immediately from the explicit formula that
\(
\lambda_\tau=\tfrac{2}{1+\sqrt{1+4E_\tau/\Lambda_{\tau-1}}}
\)
and
\(
\overline{\lambda}_\tau=\tfrac{2}{1+\sqrt{1+4qE_\tau/\Lambda_{\tau-1}}}.
\)

The condition \(E_{\tau+1}=E_\tau\) in part~\ref{item:lambda_same_E_ratio} implies that
\(
E_\tau(\lambda_{\tau+1})^2
=
E_\tau(\lambda_\tau)^2(1-\lambda_{\tau+1}),
\)
i.e.,
\(
\lambda_{\tau+1}
=
\tfrac{2(\lambda_\tau)^2}{(\lambda_\tau)^2+\sqrt{(\lambda_\tau)^4+4(\lambda_\tau)^2}}
=
\tfrac{2\lambda_\tau}{\lambda_\tau+\sqrt{(\lambda_\tau)^2+4}}
\ge
(\tfrac{2}{1+\sqrt{5}})\lambda_\tau,
\)
where the last inequality holds since $\lambda_\tau\le 1$ by Part~\ref{item:lambda_in_01_Lambda_decrease}.

We prove part~\ref{item:Gamma_k_ge_L1_over_ksq} by induction. It holds at $\tau=1$, since $\Lambda_1=E_1(\lambda_1)^2=E_1$.
Assume that $\Lambda_{\tau-1}\ge E_1/(\tau-1)^2$ for some $\tau\ge2$. By Part~\ref{item:lambda_bound}, we have $\lambda_\tau\le 2/(\tau+1)$, which leads to
$1-\lambda_\tau \ge (\tau-1)/(\tau+1)$. Hence,
\(
\Lambda_\tau
=
\Lambda_{\tau-1}(1-\lambda_\tau)
\ge
E_1/\tau^2.
\)
\end{proof}

The following lemma (see proof of Lemma 2 in \cite{nesterov2015universal}) is needed by Proposition~\ref{thm:universal_estimates_bound}.
\begin{lemma}\label{thm:universal_M_kt_bound}
Suppose that $f(\cdot)$ satisfies \eqref{eq:holder_descent}, and
$
\hat{M}\ge [ \tfrac{1-\nu}{1+\nu} ( \tfrac{1}{\delta }) ]^{ \tfrac{1-\nu}{1+\nu}}M_\nu^{ \tfrac{2}{1+\nu}}
$, for some $\delta>0$.
Then
$
f(x) \le f(u) +\langle f'(u), x-u\rangle +  \tfrac{\hat{M}}{2}\|x-u\|^2 +  \tfrac{\delta}{2}, \quad \forall x,u\in X.
$
\end{lemma}

\begin{proof}[Proof of Proposition~\ref{thm:universal_estimates_bound}]
Since $L_0\le L$ and $L_k=L_k^{S_k}$, the description of $L_k^s$ indicates $L_k\le 2L$.
To prove the bound of $M_k^{1,s}$, note that our assumption of this proposition and the definition of $M_k^{0,s}$ in Algorithm~\ref{alg:universal_linesearch} guarantee that
$M_{k}^{0,s}
= 
\max\{ M_0, L_k^s\}
\le
\max\{ M_0, 2L\}
\le  
2[\tfrac{1-\nu}{1+\nu}(\tfrac{1}{\gamma_k^s\varepsilon})]^{\tfrac{1-\nu}{1+\nu}}M_\nu^{\tfrac{2}{1+\nu}}.
$
Here the last inequality holds due to $\gamma_k^s\le 1$ by Lemma~\ref{thm:recursion_property}\ref{item:lambda_in_01_Lambda_decrease}.
Since $M_k^{1,s}=2^{i_k^{1,s}}M_k^{0,s}$ where $i_k^{1,s}$ is the smallest nonnegative integer such that \eqref{eq:universal_f_ineq_s} holds, Lemma \ref{thm:universal_M_kt_bound} gives the desired bound of $M_k^{1,s}\le 2 [ \tfrac{1-\nu}{1+\nu}  ( \tfrac{1}{\gamma_k^s\alpha_k^{1,s}\varepsilon} ) ]^{ \tfrac{1-\nu}{1+\nu}}M_\nu^{ \tfrac{2}{1+\nu}}$.
Lemma~\ref{thm:recursion_property}\ref{item:lambda_decrease} therefore gives
$M_k^{1,s}\le 2 [ \tfrac{1-\nu}{1+\nu}  ( \tfrac{1}{\gamma_k^s\alpha_k^{2,s}\varepsilon} ) ]^{ \tfrac{1-\nu}{1+\nu}}M_\nu^{ \tfrac{2}{1+\nu}}$. 
Moreover, for any \(1\le t-1\le T_k^s-1\), the first trial in the
\(t\)-th inner iteration is \(c_k^{t-1,s}M_k^{t-1,s}\). If \(c_k^{t-1,s}=1\),
then the desired bound follows from the induction hypothesis and
\(\alpha_k^{t,s}\le \alpha_k^{t-1,s}\). If \(c_k^{t-1,s}>1\), then by the
definition of \(c_k^{t-1,s}\), the first trial satisfies
\(
    c_k^{t-1,s}M_k^{t-1,s}(\alpha_k^{t,s})^2=L_k^s.
\)
If this first trial were  larger than the threshold in
Lemma~\ref{thm:universal_M_kt_bound}, then \eqref{eq:universal_f_ineq_s}
would hold and the subroutine would terminate at iteration \(t\), contradicting
\(t\le T_k^s-1\). Therefore the first trial  is below the threshold, and
the doubling search gives
\(
M_k^{t,s}
\le
2[
\tfrac{1-\nu}{1+\nu}
(
\tfrac{1}{\gamma_k^s\alpha_k^{t,s}\varepsilon}
)
]^{\tfrac{1-\nu}{1+\nu}}
M_\nu^{\tfrac{2}{1+\nu}} .
\)
Performing this induction  concludes the proof.
\end{proof}
\begin{proof}[Proof of Proposition~\ref{thm:universal_GammaA_bound}]
Recall in \eqref{eq:AGamma} that $\Gamma_N^{S_N} = L_N(\gamma_N^{S_N})^2$. By Proposition \ref{thm:universal_estimates_bound} we have $L_N\le 2L$ and by Lemma \ref{thm:recursion_property}\ref{item:lambda_bound} we have $\gamma_N^{S_N}\le 2/(N+1)$. Therefore $\Gamma_N^{S_N}\le 8L/(N+1)^2$ holds. To prove \eqref{eq:Gamma_bound}, we proceed in three steps. 
First, we establish a relation between $\gamma_k^s/(\Gamma_k^s)^{\tfrac{1+\nu}{1+3\nu}}$ and $T_k^s$. Specifically, fix any $k,s\ge1$, by the bound on $A_k^{T_k^s,s}$ and the identity $A_k^{T_k^s,s}=L_k^s$ from Proposition~\ref{thm:Ak1s_ge_Lks} and the relation $\Gamma_k^s=L_k^s(\gamma_k^s)^2$, we have
\(
\tfrac{(\gamma_k^s)^2}{\Gamma_k^s}
=
\tfrac{1}{L_k^s}
=
\tfrac{1}{A_k^{T_k^s,s}}
\ge
 (\tfrac{T_k^s}{2} )^{\tfrac{1+3\nu}{1+\nu}}
\tfrac{1}{2}\,
(\overline{M}_{\nu,\varepsilon})^{-1}
(\gamma_k^s)^{\tfrac{1-\nu}{1+\nu}},
\)
i.e.,
{\small
\begin{equation}\label{eq:T_ks_bound_by_gammaGamma}
T_k^s
\le
 2(2\overline{M}_{\nu,\varepsilon})^{\tfrac{1+\nu}{1+3\nu}}\gamma_k^s/(\Gamma_k^s)^{\tfrac{1+\nu}{1+3\nu}}.
\end{equation}
}
Hence, for all $k\ge1$,
{\small
\begin{equation}\label{eq:sum_T_bound_in_down_up_searching_L}
     \sum\nolimits_{s=1}^{S_k}T_k^s
    \le
2(2\overline{M}_{\nu,\varepsilon})^{\tfrac{1+\nu}{1+3\nu}}
\sum\nolimits_{s=1}^{S_k}\gamma_k^s/(\Gamma_k^s)^{\tfrac{1+\nu}{1+3\nu}}.
\end{equation}
}
Next, we further bound \( \sum\nolimits_{s=1}^{S_k} T_k^s\) for each $k\ge1$ via the right-hand side of \eqref{eq:sum_T_bound_in_down_up_searching_L}.
For any $k\ge1$ and $s\ge1$, Algorithm~\ref{alg:universal_linesearch} ensures $L_k^{s+1}=2L_k^s$. Hence, by Lemma~\ref{thm:recursion_property}\ref{item:lambda_compare_L_increase}, we have
\(
\gamma_k^s\le \sqrt{2}\,\gamma_k^{s+1},
\)
and consequently
{\small \begin{equation}\label{eq:gamma_over_Gamma_change}
\begin{aligned}
\tfrac{\gamma_k^{s+1}}{(\Gamma_k^{s+1})^{\tfrac{1+\nu}{1+3\nu}}}
&=
\tfrac{(\gamma_k^{s+1})^{\tfrac{\nu-1}{1+3\nu}}}{(L_k^{s+1})^{\tfrac{1+\nu}{1+3\nu}}}
\le
\tfrac{(\sqrt{2})^{-\tfrac{\nu-1}{1+3\nu}}(\gamma_k^s)^{\tfrac{\nu-1}{1+3\nu}}}{(2L_k^s)^{\tfrac{1+\nu}{1+3\nu}}}
=
\tfrac{(\gamma_k^s)^{\tfrac{\nu-1}{1+3\nu}}}{\sqrt{2}(L_k^s)^{\tfrac{1+\nu}{1+3\nu}}}
=
\tfrac{\gamma_k^s}{\sqrt{2}(\Gamma_k^s)^{\tfrac{1+\nu}{1+3\nu}}},
\end{aligned}
\end{equation} }
where the inequality follows from the negative exponent
\(
\tfrac{\nu-1}{1+3\nu}
\).
Combining \eqref{eq:sum_T_bound_in_down_up_searching_L} with the geometric decay in \eqref{eq:gamma_over_Gamma_change}, we obtain for all $k\ge1$,
{\small
\begin{align}\label{eq:universal_sumT_bound}
 \sum\nolimits_{s=1}^{S_k}T_k^s
\le
\tfrac{2}{(\tfrac{1}{2}\overline{M}_{\nu,\varepsilon}^{-1})^{\tfrac{1+\nu}{1+3\nu}}}
 (\tfrac{\sqrt{2}}{\sqrt{2}-1} )
\tfrac{\gamma_k^1}{(\Gamma_k^1)^{\tfrac{1+\nu}{1+3\nu}}}
=
\tfrac{2}{(\tfrac{1}{2}\overline{M}_{\nu,\varepsilon}^{-1})^{\tfrac{1+\nu}{1+3\nu}}}
 (\tfrac{\sqrt{2}}{\sqrt{2}-1} )
\tfrac{(\gamma_k^1)^{\tfrac{\nu-1}{1+3\nu}}}{(L_k^1)^{\tfrac{1+\nu}{1+3\nu}}}.
\end{align}
}
Specifically, we have $\gamma_1^1=1$ and $L_1^1=L_0$, hence \eqref{eq:universal_sumT_bound} reduces to \eqref{eq:universal_sumT_bound_k1}.
Moreover, for $k\ge2$, utilizing \eqref{eq:universal_sumT_bound}, relations \(L_k^1=L_{k-1}^{S_{k-1}}\) enforced by Algorithm~\ref{alg:universal_linesearch}, $\gamma_{k-1}^{S_{k-1}} 2/(1+\sqrt{5}) \le
\gamma_k^1$ by  Lemma~\ref{thm:recursion_property}\ref{item:lambda_same_E_ratio},
together with
$\Gamma_k^s = L_k^s(\gamma_k^s)^2$,
we have for $k\ge2$
{\small
\begin{align}\label{eq:universal_sumT_bound_kge2}
\begin{aligned}
 \sum\nolimits_{s=1}^{S_k}T_k^s
&\le
\tfrac{2}{(\tfrac{1}{2}\overline{M}_{\nu,\varepsilon}^{-1})^{\tfrac{1+\nu}{1+3\nu}}}
 (\tfrac{\sqrt{2}}{\sqrt{2}-1} )
 (\tfrac{2}{1+\sqrt{5}} )^{\tfrac{\nu-1}{1+3\nu}}
\tfrac{\gamma_{k-1}^{S_{k-1}}}{(\Gamma_{k-1}^{S_{k-1}})^{\tfrac{1+\nu}{1+3\nu}}}.
\end{aligned}
\end{align}
}
Last, we prove \eqref{eq:Gamma_bound} by creating a relation between $ \sum\nolimits_{s=1}^{S_k}T_k^s$, $k\ge2$ and $\Gamma_k^{S_k}$ sequence.
Define a placeholder \(1/\Gamma_0^{S_0}:=0\).
By  Lemma~\ref{thm:recursion_property}\ref{item:lambda_in_01_Lambda_decrease}, we have
\(
0<\Gamma_k^{S_k}<\Gamma_{k-1}^{S_{k-1}},
\)
which implies that for all  $k\ge1$,
{\small 
\begin{equation}\label{eq:Gammak_Gammakm1}
(
\tfrac{1}{(\Gamma_k^{S_k})^{\tfrac{1+\nu}{1+3\nu}}}
-
\tfrac{1}{(\Gamma_{k-1}^{S_{k-1}})^{\tfrac{1+\nu}{1+3\nu}}}
 )
 (
\tfrac{1}{(\Gamma_k^{S_k})^{\tfrac{2\nu}{1+3\nu}}}
+
\tfrac{1}{(\Gamma_{k-1}^{S_{k-1}})^{\tfrac{2\nu}{1+3\nu}}}
 )
>
\tfrac{1}{\Gamma_k^{S_k}}-\tfrac{1}{\Gamma_{k-1}^{S_{k-1}}}.
\end{equation}}
Utilizing the relations in \eqref{eq:universal_sumT_bound_kge2}, \eqref{eq:Gammak_Gammakm1}, \eqref{eq:AGamma_recur} and $\Gamma_k^{S_k}<\Gamma_{k-1}^{S_{k-1}}$ yields
{\small
\begin{align}\label{eq:universal_sumT_Gamma_sequence_kge2}
\begin{aligned}
&1/(\Gamma_k^{S_k})^{\tfrac{1+\nu}{1+3\nu}}
-
1/(\Gamma_{k-1}^{S_{k-1}})^{\tfrac{1+\nu}{1+3\nu}}
\ge
\tfrac{1}{2}\gamma_k^{S_k}/(\Gamma_k^{S_k})^{\tfrac{1+\nu}{1+3\nu}}
\\
\ge&
 ( \sum\nolimits_{s=1}^{S_{k+1}}T_{k+1}^s )
\tfrac{1}{4}\,
2^{-\tfrac{1+\nu}{1+3\nu}}
\overline{M}_{\nu,\varepsilon}^{-\tfrac{1+\nu}{1+3\nu}}
 (\tfrac{\sqrt{2}-1}{\sqrt{2}} )
 (\tfrac{2}{1+\sqrt{5}} )^{\tfrac{1-\nu}{1+3\nu}}, \forall k\ge1.
\end{aligned}
\end{align}
}
Hence, by utilizing \eqref{eq:universal_sumT_Gamma_sequence_kge2} and the placeholder $1/\Gamma_0^{S_0}=0$, we have 
{\small
\begin{align*}
    1/(\Gamma_{N-1}^{S_{N-1}})^{\tfrac{1+\nu}{1+3\nu}}
    =&
     \sum\nolimits_{k=1}^{N-1}
    1/(\Gamma_k^{S_k})^{\tfrac{1+\nu}{1+3\nu}}
    -
    1/(\Gamma_{k-1}^{S_{k-1}})^{\tfrac{1+\nu}{1+3\nu}}
    \\
    \ge&\,
     ( \sum\nolimits_{k=2}^N \sum\nolimits_{s=1}^{S_k}T_k^s )
    \tfrac{1}{4}\,
    2^{-\tfrac{1+\nu}{1+3\nu}}
    \overline{M}_{\nu,\varepsilon}^{-\tfrac{1+\nu}{1+3\nu}}
     (\tfrac{\sqrt{2}-1}{\sqrt{2}} )
     (\tfrac{2}{1+\sqrt{5}} )^{\tfrac{1-\nu}{1+3\nu}}.
\end{align*}
}
Then, \eqref{eq:Gamma_bound} follows by rearranging terms.
\end{proof}
\begin{proof}[Proof of Lemma~\ref{thm:additional_feval_dueto_M}]
    Fix any $k\ge1$ and $s\ge1$. By the description of Subroutine~\ref{alg:universal_subroutine_linesearch} and the fact $c_k^{t,s}\ge1$ by Lemma~\ref{thm:recursion_property}\ref{item:c_def}, we have
    $M_k^{t,s}= 2^{i_k^{t,s}}c_k^{t-1,s}M_k^{t-1,s}\ge 2^{i_k^{t,s}}M_k^{t-1,s}$, i.e., $i_k^{t,s}\le \log_2(M_k^{t,s}/M_k^{t-1,s})$,  for all $t\ge1$. Moreover, the description in Subroutine~\ref{alg:universal_subroutine_linesearch} and Lemma~\ref{thm:recursion_property}\ref{item:lambda_Lambda_E} and \ref{item:c_def} indicate that  $i_k^{T_k^s,s}=0$. So the number of (sub)gradient evaluations of $f$ from $t=1$ to $T_k^s$ is
        \(
         \sum\nolimits_{t=1}^{T_k^s}(i_k^{t,s}+1)
        =
         [ \sum\nolimits_{t=1}^{T_k^s-1}(i_k^{t,s}+1)  ]+1
        \le
        T_k^s + \log_2(M_k^{T_k^s-1,s}/M_k^{0,s}).
        \)
        Hence, after running $N$ iterations in Algorithm~\ref{alg:universal_linesearch}, the number of (sub)gradient evaluations of $f$ is
$
 \sum\nolimits_{k=1}^N \sum\nolimits_{s=1}^{S_k} \sum\nolimits_{t=1}^{T_k^s} (i_k^{t,s}+1 )
         \le \sum\nolimits_{k=1}^N \sum\nolimits_{s=1}^{S_k}T_k^s
         +
          \sum\nolimits_{k=1}^N \sum\nolimits_{s=1}^{S_k} \log_2(M_k^{T_k^s-1,s}/M_k^{0,s}).
$
It remains to bound the logarithmic term on the right-hand side of the above relation. 
Notice that if $T_k^s=1$, then $\log_2M_k^{T_k^s-1,s}/M_k^{0,s}=0$, so we only need to handle the case when $T_k^s\ge2$. By the upper bound on \(M_k^{t,s}\) established in Proposition~\ref{thm:universal_estimates_bound} and the relation \(A_k^{t,s}=M_k^{t,s}(\alpha_k^{t,s})^2\), we obtain
\(
    A_k^{T_k^s-1,s}/ (\alpha_k^{T_k^s-1,s} )^2
    =
    M_k^{T_k^s-1,s}
    \le
    2\,\overline{M}_{\nu,\varepsilon}
     (\gamma_k^s\alpha_k^{T_k^s-1,s} )^{-\tfrac{1-\nu}{1+\nu}},
\)
equivalently,
$
     (\tfrac{1}{\alpha_k^{T_k^s-1,s}} )^2
    \le
     (\tfrac{1}{A_k^{T_k^s-1,s}} )^{\tfrac{2(1+\nu)}{3\nu+1}}
     (2\,\overline{M}_{\nu,\varepsilon} )^{\tfrac{2(1+\nu)}{3\nu+1}}
     (\tfrac{1}{\gamma_k^s} )^{\tfrac{2(1-\nu)}{3\nu+1}}.
$
Then, by utilizing $A_k^{t,s} = M_k^{t,s}(\alpha_k^{t,s})^2$, the above inequality, relation $A_k^{T_k^s-1,s}\ge L_k^s$ guaranteed by Proposition~\ref{thm:Ak1s_ge_Lks}, $\Gamma_k^s=L_k^s(\gamma_k^s)^2$ in order, we have for $k\ge1$ and $s\ge1$,
{\small
\begin{equation}\label{eq:M_Tksm1_bound}
\begin{split}
    &\tfrac{M_k^{T_k^s-1,s}}{(2\,\overline{M}_{\nu,\varepsilon} )^{\tfrac{2(1+\nu)}{3\nu+1}}}
    =
    \tfrac{A_k^{T_k^s-1,s}}{(2\,\overline{M}_{\nu,\varepsilon} )^{\tfrac{2(1+\nu)}{3\nu+1}} (\alpha_k^{T_k^s-1,s} )^2}
    \le
     (\tfrac{1}{A_k^{T_k^s-1,s}} )^{\tfrac{1-\nu}{3\nu+1}}
     (\tfrac{1}{\gamma_k^s} )^{\tfrac{2(1-\nu)}{3\nu+1}}
\\
\le & 
     (\tfrac{1}{L_k^s(\gamma_k^s)^2} )^{\tfrac{1-\nu}{3\nu+1}}
    = 
     (\tfrac{1}{\Gamma_k^s} )^{\tfrac{1-\nu}{3\nu+1}}    
     \le
     (\tfrac{k^2}{L_0} )^{\tfrac{1-\nu}{3\nu+1}}.
\end{split}
\end{equation}
 }
Here the last inequality follows from
$\Gamma_k^s\ge L_1/k^2$, for $k\ge2$, by
Lemma~\ref{thm:recursion_property}\ref{item:Gamma_k_ge_L1_over_ksq} and $L_1\ge L_0$ by Algorithm~\ref{alg:universal_linesearch};
for $k=1$, it follows from $\Gamma_1^s=L_1^s\ge L_0$.
Moreover, by the relation  $M_k^{0,s}\ge L_k^s\ge L_0$ ensured in Algorithm~\ref{alg:universal_linesearch}, we obtain for $k\ge1$ and $s\ge1$,
\(
    1/M_k^{0,s}
\le
1/L_0.
\)
Combining the above relations, we have for all $k\ge1$ and $s\ge1$, 
\(
\log_2 M_k^{T_k^s-1,s}/M_k^{0,s}
\le
\tfrac{2(1+\nu)}{3\nu+1}
    \log_2 ((\tfrac{2 \overline{M}_{\nu,\varepsilon}}{L_0})k^{\tfrac{1-\nu}{1+\nu}} ).
    \)
\end{proof}

\begin{proof}[Proof of Proposition~\ref{thm:universal_GammaA_bound_final_version}]
Algorithm~\ref{alg:universal_linesearch_decreas_increase} differs from
Algorithm~\ref{alg:universal_linesearch} only in linesearch order used in the first iteration. Therefore, the proof largely follows that of
Proposition~\ref{thm:universal_GammaA_bound}; we only highlight the necessary modifications.
At $k=1$, $\gamma_1^s=1$, so $\Gamma_1^{s+1}=L_1^{s+1}=L_1^{s}/2=\Gamma_1^{s}/2$, for all $s\ge1$.
Then, it follows that 
\(
\gamma_1^{s}/ (\Gamma_1^{s} )^{\tfrac{1+\nu}{1+3\nu}}
=
\gamma_1^{s+1}/ (2\Gamma_1^{s+1} )^{\tfrac{1+\nu}{1+3\nu}}.
\)
From this geometric decay and \eqref{eq:T_ks_bound_by_gammaGamma}, it follows that 
{\small
\begin{align}\label{eq:pfugs_sumT1s}
\begin{aligned}
    \sum\nolimits_{s=1}^{S_1+1}T_1^s
    \le&
     (2/(\tfrac{1}{2}\overline{M}_{\nu,\varepsilon}^{-1})^{\tfrac{1+\nu}{1+3\nu}} )
\sum\nolimits_{s=1}^{S_1+1}\gamma_1^s/ (\Gamma_1^s )^{\tfrac{1+\nu}{1+3\nu}}
    \\
    \le&
     (\tfrac{2}{(\tfrac{1}{2}\overline{M}_{\nu,\varepsilon}^{-1})^{\tfrac{1+\nu}{1+3\nu}}} )
     (
\tfrac{1}{1-(2)^{-\tfrac{1+\nu}{1+3\nu}}}
+
2^{\tfrac{1+\nu}{1+3\nu}}
 )\tfrac{\gamma_1^{S_1}}{ (\Gamma_1^{S_1} )^{\tfrac{1+\nu}{1+3\nu}}}.
 \end{aligned}
\end{align} 
}
Using \eqref{eq:pfugs_sumT1s} and the fact $\gamma_1^s=1$, we have that
\(
1 /(\Gamma_1^{S_1} )^{\tfrac{1+\nu}{1+3\nu}}
\ge
\tfrac{ \sum\nolimits_{s=1}^{S_1+1} T_1^s}{2}
 (2^{-\tfrac{2+2\nu}{1+3\nu}}-2^{-\tfrac{3+3\nu}{1+3\nu}} )
\overline{M}_{\nu,\varepsilon}^{-\tfrac{1+\nu}{1+3\nu}}.
\)
Then, with the help of placeholder $\tfrac{1}{\Gamma_0^{S_0}}=0$ and \eqref{eq:universal_sumT_Gamma_sequence_kge2}, we have
{\small
\begin{align*}
    &2/ (\Gamma_{N-1}^{S_{N-1}} )^{\tfrac{1+\nu}{1+3\nu}}
    =
    2/ (\Gamma_1^{S_1} )^{\tfrac{1+\nu}{1+3\nu}} + 2 \sum\nolimits_{k=2}^{N-1}  1/ (\Gamma_k^{S_k} )^{\tfrac{1+\nu}{1+3\nu}}-1/ (\Gamma_{k-1}^{S_{k-1}} )^{\tfrac{1+\nu}{1+3\nu}}
   \\
   \ge &
    1/ (\Gamma_1^{S_1} )^{\tfrac{1+\nu}{1+3\nu}} +  \sum\nolimits_{k=1}^{N-1} 1/ (\Gamma_k^{S_k} )^{\tfrac{1+\nu}{1+3\nu}}-1/ (\Gamma_{k-1}^{S_{k-1}} )^{\tfrac{1+\nu}{1+3\nu}}
   \\
   \ge&
(\tfrac{2-\sqrt2}{8})
(2)^{-\tfrac{2\nu}{1+3\nu}}
(1+\sqrt5)^{-\tfrac{1-\nu}{1+3\nu}}
\overline{M}_{\nu,\varepsilon}^{-\tfrac{1+\nu}{1+3\nu}}  ( \sum\nolimits_{s=1}^{S_1+1}T_1^s+ \sum\nolimits_{k=2}^N \sum\nolimits_{s=1}^{S_k}T_k^s ).
\end{align*}}
\end{proof}
\begin{proof}[Proof of Lemma~\ref{thm:additional_feval_dueto_M_alg_final_version}]
Since this proof closely follows that of Lemma~\ref{thm:additional_feval_dueto_M};
we only highlight the necessary modifications.
By similar arguments in the proof of Lemma~\ref{thm:additional_feval_dueto_M} and
noticing $T_1^1=1$,
to complete $N$ iterations in Algorithm~\ref{alg:universal_linesearch_decreas_increase}, the number of evaluations of $f'$ is at most 
{\small
\begin{align}\label{eq:pfugs_total_feval}
\begin{aligned}
&1
+
 \sum\nolimits_{s=2}^{S_1+1} \sum\nolimits_{t=1}^{T_1^s} (i_1^{t,s}+1 )
+
 \sum\nolimits_{k=2}^N \sum\nolimits_{s=1}^{S_k} \sum\nolimits_{t=1}^{T_k^s} (i_k^{t,s}+1 )
\\
=&
 \sum\limits_{s=2}^{S_1+1}
\log_2\tfrac{M_1^{T_1^{s}-1,s}}{M_1^{0,s}}
+
 \sum\limits_{s=1}^{S_1+1}T_1^{s}
        +
         \sum\limits_{k=2}^N \sum\limits_{s=1}^{S_k}T_k^s
         +
          \sum\limits_{k=2}^N \sum\limits_{s=1}^{S_k} \log_2\tfrac{M_k^{T_k^s-1,s}}{M_k^{0,s}}.
\end{aligned}
\end{align}
}
Here the single $1$ in \eqref{eq:pfugs_total_feval} represents the number of evaluations of $f'$ needed at $k=1$, $s=1$. This is true because that $\xlb_1^{1,1} = x_0$ and hence only $f'(x_0)$ is evaluated at $k=1$ and $s=1$.
By similar arguments in \eqref{eq:M_Tksm1_bound}, we have 
$M_k^{T_k^s-1,s}
\le 
    (2\,\overline{M}_{\nu,\varepsilon} )^{\tfrac{2(1+\nu)}{3\nu+1}}
     (\tfrac{1}{\Gamma_k^s} )^{\tfrac{1-\nu}{3\nu+1}}    
     \le
 (2\,\overline{M}_{\nu,\varepsilon} )^{\tfrac{2(1+\nu)}{3\nu+1}}
     (\tfrac{2k^2}{L_1} )^{\tfrac{1-\nu}{3\nu+1}}.
$
The last inequality above holds for $k\ge2$, $s\ge1$ due to $\Gamma_k^s\ge L_1/k^2 \ge L_1/(2k^2)$ by Lemma~\ref{thm:recursion_property}\ref{item:Gamma_k_ge_L1_over_ksq} and holds for $k=1$, $s\le S_1+1$ since $\Gamma_1^s=L_1^s\ge L_1/2$  ensured by Algorithm~\ref{alg:universal_linesearch_decreas_increase}. 
From the above relation and the fact $1/M_k^{0,s}\le 1/L_k^s \le 2/L_1$, we have
{\small
\begin{equation}\label{eq:universal_log_over_log}    
\log_{2}\tfrac{M_k^{T_k^s-1,s}}{M_k^{0,s}}
\le 
    \tfrac{2(1+\nu)}{3\nu+1}
    \log_2 ((\tfrac{4 \overline{M}_{\nu,\varepsilon}}{L_1})k^{\tfrac{1-\nu}{1+\nu}} ), \;\forall k,s  
\end{equation}}
This completes the proof.
\end{proof}

\end{document}